\documentclass[12pt]{amsart}

\usepackage{graphicx}
\usepackage{latexsym}
\usepackage{amsfonts}
\usepackage{amscd}
\usepackage{amssymb}
\usepackage{amsmath}
\usepackage{amsthm}
\usepackage{bm}
\usepackage[all]{xy}
\usepackage[lmargin=1.25in,rmargin=1.25in,tmargin=1.25in,bmargin=1in,dvips]{geometry}
\usepackage{pinlabel}
\usepackage{paralist}
\usepackage{enumitem}
\usepackage[unicode, linktocpage, bookmarksnumbered=true,backref=true]{hyperref}

\usepackage[textsize=scriptsize]{todonotes}

\newtheorem{theorem}{Theorem}[section]
\newtheorem{lemma}[theorem]{Lemma}

\newtheorem{proposition}[theorem]{Proposition}

\theoremstyle{definition}
\newtheorem{definition}[theorem]{Definition}

\newtheorem{remark}[theorem]{Remark}

\newtheorem{example}[theorem]{Example}

\theoremstyle{remark}
\newtheorem{claim}{Claim}
\newtheorem*{claim*}{Claim}
\newcommand{\invlim}{\underleftarrow{\lim}}

\def\hgen{y}

\def\R{\mathbb{R}}

\def\Z{\mathbb{Z}}
\def\Q{\mathbb{Q}}
\def\F{\mathbb{F}}

\def\CF {\operatorname{CF}}
\def\HF {\operatorname{HF}}

\def\CFp {\operatorname{CF}^+}
\def\CFm {\operatorname{CF}^-}
\def\CFi {\operatorname{CF}^\infty}
\def\HFp {\operatorname{HF}^+}
\def\HFm {\operatorname{HF}^-}
\def\HFi {\operatorname{HF}^\infty}

\def\HFic {\mathbf{HF}^\infty}
\def\CFic {\mathbf{CF}^\infty}

\def\CFK {\operatorname{CFK}}

\newcommand{\Spinc}{\operatorname{Spin^c}}
\def\spincs {\mathfrak{s}}
\def\spinct {\mathfrak{t}}

\def\AA {\mathcal{A}}

\def\CC {\mathcal{C}}

\def\KK {\mathcal{K}}
\def\LL {\mathcal{L}}
\def\MM {\mathcal{M}}

\def\QQ {\mathcal{Q}}
\def\RR {\mathcal{R}}

\def\T{\mathbb{T}}

\def\s {\mathbf{s}}
\def\k {\mathbf{k}}

\def\x {\mathbf{x}}
\def\y {\mathbf{y}}

\def\ul {\underline}

\newcommand{\abs}[1] {\left\lvert #1 \right\rvert}

\newcommand{\floor}[1] {\left\lfloor #1 \right\rfloor}
\newcommand{\gen}[1] {\langle #1 \rangle}
\def\Th{^{\text{th}}}
\def\minus{\smallsetminus}
\def\co{\colon\thinspace}

\def\dtop{d_{\operatorname{top}}}
\def\dbot{d_{\operatorname{bot}}}
\def\dtw{d}
\def\dsh{\tilde d}
\def\dfour{\dsh}
\def\dknot{\dsh}

\def\conn{\mathbin{\#}}

\DeclareMathOperator{\im}{im}  \DeclareMathOperator{\rank}{rank}
  \DeclareMathOperator{\pt}{pt}
 \DeclareMathOperator{\Spin}{Spin}
\DeclareMathOperator{\Span}{Span} \DeclareMathOperator{\nbd}{nbd}
\DeclareMathOperator{\PD}{PD} 
\DeclareMathOperator{\Tors}{Tors} \DeclareMathOperator{\Tor}{Tor}
\DeclareMathOperator{\Hom}{Hom} \DeclareMathOperator{\Ext}{Ext}

\newcommand{\HomRing}[1]{\mathcal{H}_{#1}}

\definecolor{darkblue}{rgb}{0,0,0.5}
\definecolor{darkred}{rgb}{0.5,0,0}
\definecolor{darkgreen}{rgb}{0,0.5,0}

\numberwithin{equation}{section}

\def\ul{\underline}

\def\Hc{H_c}
\def\Hlf{H^{\mathit{lf}}}
\def\Clf{C^{\mathit{lf}}}

\begin{document}

\title{Heegaard Floer invariants in codimension one}
\author[Adam Simon Levine]{Adam Simon Levine}
\address{Department of Mathematics \\ Princeton University \\ Princeton, NJ 08540}
\email{asl2@math.princeton.edu}
\thanks{ASL was partially supported by NSF grant DMS-1405378.}

\author[Daniel Ruberman]{Daniel Ruberman}
\address{Department of Mathematics, MS 050 \\ Brandeis University \\ Waltham, MA 02454}
\email{ruberman@brandeis.edu}
\thanks{DR was partially supported by NSF grant DMS-1506328.}
\date{October 23, 2016}

\begin{abstract}
Using Heegaard Floer homology, we construct a numerical invariant for any smooth, oriented $4$-manifold $X$ with the homology of $S^1 \times S^3$. Specifically, we show that for any smoothly embedded $3$-manifold $Y$ representing a generator of $H_3(X)$, a suitable version of the Heegaard Floer $d$ invariant of $Y$, defined using twisted coefficients, is a diffeomorphism invariant of $X$. We show how this invariant can be used to obstruct embeddings of certain types of $3$-manifolds, including those obtained as a connected sum of a rational homology $3$-sphere and any number of copies of $S^1 \times S^2$. We also give similar obstructions to embeddings in certain open $4$-manifolds, including exotic $\R^4$s.
\end{abstract}
\maketitle

\section{Introduction} \label{sec: introduction}

A powerful way to study a non-simply connected manifold $X$ is to look at invariants of a codimension $1$ submanifold $Y$ dual to an element of $H^1(X)$. This idea goes back to work of Pontrjagin, Rohlin, and Novikov in the 1950s and 1960s exploring ``codimension $1$ (and higher) signatures'' (see \cite{rohlin:pontrjagin, novikov}).  In dimension $4$, if $X$ has the homology of $S^1 \times S^3$, then the Rohlin invariant of a submanifold $Y$ representing a generator of $H_3(X)$, with spin structure induced from $X$, is a diffeomorphism invariant of $X$~\cite{ruberman:ds}. (We call $Y$ a \emph{cross-section} of $X$.) This invariant has interpretations in terms of Seiberg--Witten theory~\cite{mrowka-ruberman-saveliev:sw-index} and conjecturally in terms of Yang--Mills theory~\cite{ruberman-saveliev:casson, ruberman-saveliev:mappingtori, ruberman-saveliev:4tori, ruberman-saveliev:survey}. More recently, Fr{\o}yshov~\cite{froyshov:qhs-floer} observed that if $X$ has a cross-section $Y$ that is a rational homology $3$-sphere, the invariant $h(Y,\spincs_X) \in \Q$ associated to the unique $\Spinc$ structure $\spincs_X$ on $Y$ induced from $X$, which is defined using monopole Floer homology, is also a smooth invariant of $X$. Fr{\o}yshov's argument uses only the rational homology cobordism invariance property of $h(Y,\spincs_X)$, so it applies verbatim to the version of $h(Y,\spincs_X)$ defined by Kronheimer and Mrowka~\cite[\S39.1]{kronheimer-mrowka:monopole} (presumed, but not known, to be equal to Fr{\o}yshov's) and the similarly defined Heegaard Floer correction term $d(Y,\spincs_X)$~\cite{OSzAbsolute}. In this paper, we extend the range of the Heegaard Floer invariant to an arbitrary smooth $4$-manifold $X$ with the homology of $S^1 \times S^3$, without the requirement that $X$ admit a rational homology sphere cross-section. Note that this is a non-trivial restriction; for instance, the Alexander polynomial obstructs the existence of such cross-sections.

The definition of the correction term $d(Y,\spincs)$ for a rational homology sphere $Y$ relies on the fact that $\HFi(Y,\spincs)$ is isomorphic to $\F[U,U^{-1}]$. (Here $\F$ denotes the field of two elements.) In our earlier work~\cite{LevineRubermanStrleNonorientable, LevineRubermanCorrection}, we showed how to extend the definition of the correction terms for manifolds with $b_1(Y)>0$ for which $\HFi(Y,\spincs)$ is ``standard.'' (Here, $\spincs$ is assumed to be a torsion spin$^c$ structure.) Work of Lidman \cite{LidmanInfinity} shows that this condition holds whenever the triple cup product on $H^1(Y;\Z)$ vanishes identically. However, an arbitrary $4$-manifold $X$ with the homology of $S^1 \times S^3$ need not have any cross-section with standard $\HFi$.

In the present paper, we use a further generalization of the correction terms. For any subspace $A$ of $H^1(Y)$ on which the triple cup product vanishes, we show in Theorem \ref{thm: HFi-standard} that the {\em twisted} Heegaard Floer homology group $\HFi(Y,\spincs; M_A)$ with coefficients in $M_A  = \F[H^1(Y)/A]$ is standard in a suitable sense, allowing us to define a twisted correction term $\dtw(Y,\spincs; M_A)$. (The case where $A=0$ has been studied by Behrens and Golla \cite{BehrensGollaCorrection}.) When $Y$ is a cross-section of $X$, we identify a particular such subspace by studying the cohomology of the infinite cyclic cover $\tilde X$. Our main result, which is stated more precisely below as Theorem \ref{thm: dfour}, is as follows:

\begin{theorem} \label{thm: dfour-intro}
Let $X$ be a homology $S^1 \times S^3$, and let $Y$ be any cross-section of $X$ representing a fixed generator $\hgen$ of $H_3(X)$. Then the correction term of $(Y,\spincs_X)$, suitably normalized, is independent of the choice of $Y$. Thus, we obtain an invariant $\dfour(X,\hgen)$, which depends only on the diffeomorphism type of $X$ and the choice of generator $\hgen \in H_3(X)$.
\end{theorem}

In principle, the invariant $\dfour(X,\hgen)$ could be used to detect exotic smooth structures on $S^1 \times S^3$, but we do not know of any candidates.

A more tractable application comes from the behavior of $\dfour(X,\hgen)$ under reversing either the orientation of $X$ or the choice of generator of $H_3(X)$. In general, the four numbers $\dfour(\pm X, \pm\hgen)$ are \emph{a priori} unrelated to each other, so they can obstruct the existence of symmetries that reverse the orientations of $X$ or $Y$. Moreover, in Section \ref{subsec: d-reversal} we describe a class of $3$-manifolds which are called \emph{$d$-symmetric}; this includes any manifold of the form $Q \conn n(S^1 \times S^2)$, where $Q$ is a rational homology sphere and $n \ge 0$. The following proposition describes some further symmetries of the invariants $\dfour(\pm X, \pm\hgen)$:

\begin{proposition} \label{prop: dfour-symmetries}
Let $X$ be a homology $S^1 \times S^3$.
\begin{itemize}
\item
If $X$ has a cross-section that is a rational homology sphere, then
\begin{equation} \label{eq: dfour-QHS}
\dfour(X,\hgen) = \dfour(-X,\hgen) = -\dfour(X,-\hgen) = -\dfour(-X,-\hgen).
\end{equation}

\item
If $X$ has a cross-section that is $d$-symmetric, then
\begin{equation} \label{eq: dfour-symmetric}
\dfour(X,\hgen) = -\dfour(-X,-\hgen) \quad \text{and} \quad \dfour(-X,\hgen) = -\dfour(X,-\hgen).
\end{equation}

\item
If $X$ is the mapping torus of a diffeomorphism $\phi \co Y \to Y$, then
\begin{equation} \label{eq: dfour-fibered}
\begin{aligned}
\dfour(X,\hgen) &= \dfour(-X,\hgen) = \ul d(Y,\spincs_X) + \frac{b_1(Y)}{2} \\
\dfour(X,-\hgen) &= \dfour(-X,-\hgen) = \ul d(-Y,\spincs_X) + \frac{b_1(Y)}{2}
\end{aligned}
\end{equation}
where $\ul d$ denotes the twisted correction term defined by Behrens and Golla \cite{BehrensGollaCorrection}.
\end{itemize}
\end{proposition}

In particular, the failure of \eqref{eq: dfour-QHS} or \eqref{eq: dfour-symmetric} for a given $4$-manifold $X$ enables us to obstruct the existence of particular types of cross-sections in $X$. In Section \ref{sec: 2knots}, we apply this obstruction to the study of $3$-dimensional Seifert surfaces for knotted $2$-spheres in $S^4$.

The proof of Theorem \ref{thm: dfour-intro} relies on examining the lift of a cross-section of $X$ to the infinite cyclic cover $\tilde X$. In fact, our techniques are more general; we consider a $d$ invariant associated to any open $4$-manifold $\tilde X$ satisfying certain homological properties similar to those of an infinite cyclic cover and any embedded $3$-manifold $Y$ representing a generator of $H_3(\tilde X)$, which we also call a cross-section. While this quantity depends on the choice of $Y$ and not just its homology class, we prove an inequality relating the invariants of disjoint cross-sections, which implies Theorem \ref{thm: dfour-intro} in the case where $\tilde X$ is actually the $\Z$ cover of a homology $S^1 \times S^3$. In the general case, the inequality still gives interesting restrictions on the types of cross-sections that can occur. As one application (Example \ref{ex: exoticR4}), we construct an exotic $\R^4$ that has no $d$-symmetric $3$-manifold sufficiently far out in its end.

\subsection*{Acknowledgments}

We are grateful to Marco Golla, Bob Gompf, Matthew Hedden, and Tye Lidman for many enlightening discussions, and to the referee for helpful suggestions.

\section{The surgery formula} \label{sec: surgery-formula}

In this section, we state a twisted version of the mapping cone formula for the Heegaard Floer homology of surgeries on knots. This formula is known to experts but does not appear in the literature; the proof is a straightforward generalization of Ozsv\'ath and Szab\'o's original integer surgery formula \cite{OSzSurgery}. We will use this formula in Section \ref{sec: twisted-d} in order to prove that $\HFi$ with appropriately twisted coefficients has a standard form.

\subsection{Heegaard Floer preliminaries} \label{sec: HF-prelim}

Throughout the paper, all Heegaard Floer homology groups are taken over the ground field $\F = \Z/2\Z$. Singular and simplicial homology and cohomology groups are taken with coefficients in $\Z$ unless otherwise specified.

We first provide a brief overview of Heegaard Floer homology with twisted coefficients. See Ozsv\'ath--Szab\'o \cite{OSzProperties} for the original definition, and Jabuka--Mark \cite{JabukaMarkProduct} for an excellent exposition. Here we emphasize two aspects of the theory that will be needed later: passing from $\HFp$ to $\HFi$ via the $U$--completed version $\HFic$, and the behavior of the coefficient modules under cobordism maps.

Let $Y$ be a closed, connected, oriented $3$--manifold, and let $\spincs$ be a spin$^c$ structure on $Y$. Let $\HomRing{Y} = \F[H^1(Y)]$; this can be identified with the ring of Laurent polynomials in $b_1(Y)$ variables.\footnote{The ring $\HomRing{Y}$ is called $R_Y$ in \cite{JabukaMarkProduct}; we use the notation $\HomRing{Y}$ to avoid confusion with the manifold $R_Y$ in Section \ref{sec: cross-section}.} Associated to $(Y,\spincs)$, there are chain complexes $\CFm(Y,\spincs;\HomRing{Y})$, $\CFi(Y,\spincs;\HomRing{Y})$, and $\CFp(Y,\spincs;\HomRing{Y})$ over $\HomRing{Y}[U]$, well-defined up to chain homotopy equivalence, which fit into a short exact sequence
\begin{equation} \label{eq: CF-exact}
0 \to \CFm(Y,\spincs; \HomRing{Y}) \xrightarrow{\iota} \CFi(Y,\spincs; \HomRing{Y}) \xrightarrow{\pi} \CFp(Y,\spincs; \HomRing{Y}) \to 0.
\end{equation}
We use $\CF^{\circ}(Y,\spincs; \HomRing{Y})$ to refer to any of the three complexes (or, by abuse of notation, the exact sequence relating them). Note that $\CF^{\circ}(Y,\spincs;\HomRing{Y})$ always has a relative $\Z$--grading, which multiplication by $U$ drops by $2$. If $\spincs$ is torsion, multiplication by any element of $\HomRing{Y}$ preserves this grading, and the grading lifts to an absolute $\Q$--grading. (When $\spincs$ is non-torsion, one must put a nontrivial grading on $\HomRing{Y}$ to define the relative $\Z$--grading, but we shall focus on torsion spin$^c$ structures throughout the paper.) If $M$ is any $\HomRing{Y}$--module, then let $\CF^{\circ}(Y,\spincs; M) = \CF^{\circ}(Y,\spincs; \HomRing{Y}) \otimes_{\HomRing{Y}} M$; these groups fit into a short exact sequence just like \eqref{eq: CF-exact}.\footnote{We will not be focusing on non-torsion spin$^c$ structures in this paper, but the fact that $\CF^{\circ}(Y,\spincs;\HomRing{Y})$ is relatively $\Z$--graded (and not just $\Z/2d\Z$--graded for some $d>0$) is one of the advantages of twisted coefficients in other settings. Moreover, depending on the choice of $M$, the $\Z$--grading sometimes descends to $\CF^\circ(Y,\spincs;M)$ even when $\spincs$ is non-torsion. See \cite[Section 3]{JabukaMarkProduct} for a nice discussion.} The homology groups are denoted by $\HF^\circ(Y,\spincs;M)$ and fit into a long exact sequence
\begin{equation} \label{eq: HF-exact}
\dots \to \HFm(Y,\spincs; M) \xrightarrow{\iota_M} \HFi(Y,\spincs; M) \xrightarrow{\pi_M} \HFp(Y,\spincs; M) \to \dots
\end{equation}
(We will frequently omit the subscripts from $\iota_M$ and $\pi_M$ unless they are needed for clarity.)

Any element $\zeta \in H_1(Y)$ induces a degree $-1$, $\HomRing{Y}[U]$--linear chain map
\[
\AA_\zeta \co \CF^{\circ}(Y,\spincs; \HomRing{Y}) \to \CF^{\circ}(Y,\spincs; \HomRing{Y}),
\]
which is well-defined up to chain homotopy. Let $\AA_\zeta^M$ denote the induced map on $\CF^\circ(Y,\spincs;M)$. These induce an action of $\Lambda^*(H_1(Y)/\Tors) \otimes \HomRing{Y}[U]$ on $\HF^\circ(Y,\spincs;M)$. Moreover, following \cite[Remark 5.2]{JabukaMarkProduct}, define
\begin{align}
\label{eq: ZM} Z_M &= \{\alpha \in H^1(Y) \mid \alpha m = m \ \forall m \in M \} \\
\label{eq: ZM-perp} Z_M^\perp &= \{\zeta \in H_1(Y) \mid \gen{\alpha, \zeta} = 0 \ \forall \alpha \in Z_M\}.
\end{align}
For any $\zeta \in Z_M^\perp$, $\AA_\zeta^M$ is chain-homotopic to $0$. Thus, the $H_1$ action descends to an action of $\Lambda^*(H_1(Y)/Z_M^\perp) \otimes \HomRing{Y}[U]$.

If $A$ is a subspace of $H^1(Y)$ such that $H^1(Y)/A$ is torsion-free (i.e., a direct summand of $H^1(Y)$), let $M_A = \F[H^1(Y)/A]$, viewed as an $\HomRing{Y}$--module via the quotient map. Concretely, if $\alpha_1, \dots, \alpha_n$ are a basis for $H^1(Y)$ such that $A = \Span(\alpha_1, \dots, \alpha_k)$, and $t_i \in \HomRing{Y}$ corresponds to $\alpha_i$, then
\[
M_A = \HomRing{Y}/(t_1-1, \dots, t_k-1).
\]
Moreover, $Z_{M_A}=A$ and $Z_M^\perp = A^\perp$, and $H_1(Y)/A^\perp$ is naturally isomorphic to the dual of $A$.

Ignoring the $\HomRing{Y}$--module structure and the $H_1$ action, we note the following basic fact:

\begin{lemma} \label{lemma: HFi-free}
Let $Y$ be a closed, oriented $3$-manifold, $\spincs$ a torsion spin$^c$ structure on $Y$, and $A \subset H^1(Y)$ a direct summand of rank $k$. Then $\HFi(Y,\spincs;M_A)$ is a free, finitely-generated $\F[U,U^{-1}]$--module with rank at most $2^k$.
\end{lemma}

\begin{proof}
Ozsv\'ath and Szab\'o \cite[Theorem 10.12]{OSzProperties} proved the $k=0$ case:
\[
\HFi(Y,\spincs;\HomRing{Y}) \cong \F[U,U^{-1}],
\]
where every element of $H^1(Y)$ acts by the identity. Thus, assume $k>0$. Up to an overall shift, assume that $\HFi(Y,\spincs;\HomRing{Y})$  is supported in even integer gradings. Consider the universal coefficient spectral sequence for changing coefficients from $\HomRing{Y}$ to $M_A$. The $E^2$ page satisfies
\begin{align*}
E^2_{p,q} &= \Tor_p^{\HomRing{Y}}\left( \HFi(Y,\spincs;\HomRing{Y}), M_A \right) \\
&=
\begin{cases}
\F^{\binom{k}{p}} & q \text{ even} \\
0 & q \text{ odd}.
\end{cases}
\end{align*}
For each $s \in \Z$, by summing over all $p,q$ with $p+q=s$, we thus deduce that $\dim_\F \HFi_s(Y,\spincs; M_A) \le 2^{k-1}$. Choose bases (over $\F$) for the summands in grading $0$ and $1$; since $\HFi(Y,\spincs; M_A)$ is relatively $\Z$-graded, these combine to give a basis for $\HFi(Y,\spincs;M_A)$ over $\F[U,U^{-1}]$.
\end{proof}

In the proof of the surgery formula (Theorem \ref{thm: mapping-cone}) below, we will need to pass from a result about $\HFp$ to a result about $\HFi$. This is best done by first considering the $U$-completed version, introduced in \cite{ManolescuOzsvathLink}. Define
\begin{align*}
\CFic(Y,\spincs;M) &= \CFi(Y,\spincs;M) \otimes_{\F[U,U^{-1}]} \F[[U,U^{-1}]
\end{align*}
Denote the homology of this complex by $\HFic(Y,\spincs;M)$. Because $\F[[U,U^{-1}]$ is flat over $\F[U,U^{-1}]$, we have
\begin{equation} \label{eq: HFic}
\HFic(Y,\spincs;M) \cong \HFi(Y,\spincs;M) \otimes_{\F[U,U^{-1}]} \F[[U,U^{-1}]
\end{equation}
Because multiplication by $U$ drops grading by $2$, it can also be understood as a grading-preserving map $\CFp(Y,\spincs;M) \to \CFp(Y,\spincs;M)[2]$.\footnote{We use the following convention: If $C$ is a graded vector space and $n \in \Q$, $C[n]$ denotes the graded vector space with $C[n]_k = C_{k-n}$.} It is easy to see that $\CFic(Y,\spincs;M)$ is isomorphic to the inverse limit of the directed system
\[
\dots \xrightarrow{U} \CFp(Y,\spincs;M)[-2] \xrightarrow{U} \CFp(Y,\spincs;M) \xrightarrow{U} \CFp(Y,\spincs;M)[2] \xrightarrow{U} \dots.
\]
For conciseness, we write
\[
\CFic(Y,\spincs;M) \cong \invlim (\CFp(Y,\spincs; M), U).
\]
There is therefore a short exact sequence
\[
0 \to \invlim^1 (\HFp(Y,\spincs;M), U)_{*-1} \to \HFic_*(Y,\spincs;M) \to \invlim (\HFp(Y,\spincs;M),U)_* \to 0
\]
(where the $*$ denotes the homological grading). The system $(\HFp(Y,\spincs;M),U)$ satisfies the Mittag-Leffler condition, since for all $n$ sufficiently large, the image of $U^n$ is equal to the image of $\pi \co \HFi(Y,\spincs;M) \to \HFp(Y,\spincs;M)$.  (See, e.g., \cite[Proposition 3.5.7]{weibel:intro}.) Thus, the derived functor $\invlim^1 (\HFp(Y,\spincs;M), U)$ vanishes, and we deduce that
\begin{equation} \label{eq: invlim-HFp}
\HFic(Y,\spincs;M) \cong \invlim (\HFp(Y,\spincs; M), U).
\end{equation}

For a non-torsion spin$^c$ structure $\spincs$, $\HFic(Y,\spincs;M)$ does not generally determine  $\HFi(Y,\spincs;M)$; see \cite[Section 2]{ManolescuOzsvathLink}. However, when $\spincs$ is torsion and $M = M_A$ for some summand $A \subset H^1(Y)$, the two theories are essentially interchangeable. Specifically, by Lemma \ref{lemma: HFi-free} and \eqref{eq: HFic}, $\HFic(Y,\spincs;M_A)$ is a finitely generated, free $\F[[U,U^{-1}]$--module whose rank (over $\F[[U,U^{-1}]$) is the same as the rank of $\HFi(Y,\spincs;M_A)$ (over $\F[U,U^{-1}])$. It is thus clear how to recover $\HFi(Y,\spincs;M_A)$ from $\HFic(Y,\spincs;M_A)$. Moreover, the action of $\HomRing{Y}$ is grading preserving, and the action of $H_1(Y)$ drops grading by $1$, so these actions on $\HFi(Y,\spincs;M_A)$ and $\HFic(Y,\spincs;M_A)$ are readily identified.

Next, we discuss the cobordism maps on twisted Heegaard Floer homology. If $W \co Y \to Y'$ is a cobordism between closed, connected, oriented $3$-manifolds, consider the exact sequence
\begin{equation} \label{eq: H(W,dW)}
H^1(\partial W) \xrightarrow{\delta_W} H^2(W, \partial W) \xrightarrow{j_W} H^2(W),
\end{equation}
and let $K(W) = \im(\delta_W) = \ker(j_W)$. The map $\delta_W$ makes $\F[K(W)]$ into an $\HomRing{Y}$--$\HomRing{Y'}$ bimodule. Given an $\HomRing{Y}$--module $M$, let $M(W) = M \otimes_{\HomRing{Y}} \F[K(W)]$. Note that the map $j_W$ is given by the intersection form on $W$; if this form vanishes (meaning there are no classes in $H_2(W)$ of nonzero square), then $K(W) = H^2(W,\partial W)$. According to \cite[Section 2.7]{OSz4Manifold}, for any spin$^c$ structure $\spinct$ on $W$, there is an induced map
\begin{equation} \label{eq: induced-module}
F^\circ_{W,\spinct} \co \HF^\circ(Y, \spinct_{Y}; M) \to \HF^\circ(Y', \spinct|_{Y'}; M(W)),
\end{equation}
which is an invariant of $(W,\spinct)$ up to multiplication by units in $\HomRing{Y}$ and $\HomRing{Y'}$.

For future reference, let us describe $\F[K(W)]$ in the case where $W$ is given by a single handle attachment. To be precise, we assume that
\[
W = Y \times [0,1] \cup \text{$k$-handle},
\]
where the handle is attached along a $(k-1)$--sphere in $Y \times \{1\}$ and $k\in \{1,2,3\}$. (Note that any connected cobordism between connected $3$-manifolds has a handle decomposition with only $1$-, $2$-, and $3$-handles.) In each of these cases, it is easy to describe $K(W)$ as an $\HomRing{Y}$--module. It is easier to work in terms of homology, identifying the sequence \eqref{eq: H(W,dW)} with
\[
H_2(\partial W) \to H_2(W) \to H_2(W,\partial W)
\]
via Poincar\'e duality.

\begin{itemize}
\item
When $k=1$, the inclusion $Y \to W$ induces an isomorphism $H_2(Y) \to H_2(W)$, and hence $K(W) = H^2(W, \partial W) \cong H^1(Y)$. Hence, for any $\HomRing{Y}$--module $M$, we have $M(W) \cong M$. Similarly, when $k=3$, if we let $t \in \HomRing{Y_1}$ denote the Poincar\'e dual of the attaching sphere (which is assumed to be nonseparating and therefore a primitive class), we see that $\F[K(W)] \cong \HomRing{Y_1}/(t-1)$, and therefore $M(W) \cong M /(t-1)M$.

\item
When $k=2$, let $K \subset Y$ denote the attaching circle for the $2$-handle. The exact sequence on homology for the pair $(W,Y)$ gives
\begin{equation} \label{eq: H(W,Y)}
0 \to H_2(Y) \to H_2(W) \to \Z \to H_1(Y),
\end{equation}
where $1 \in \Z$ maps to $[K] \in H_1(Y)$.

If $K$ is rationally null-homologous, let $d>0$ denote its order in $H_1(Y)$. A capped-off rational Seifert surface for $K$ produces a class $[\hat S] \in H_2(W)$ that maps to $d \in \Z$ in \eqref{eq: H(W,Y)}. Therefore, $H_2(W) \cong H_2(Y) \oplus \Z$; the $H_2(Y)$ summand is canonical, while the $\Z$ is generated by $[\hat S]$. If the $2$-handle is attached along a multiple of the rational longitude for $K$ (meaning that the self-intersection of $[\hat S]$ is zero), then the map $H_2(Y') \to H_2(W)$ is surjective, so $K(W) \cong H^2(W, \partial W) \cong H^1(Y) \oplus \Z$.  We thus have $\F[K(W)] \cong \HomRing{Y}[t,t^{-1}]$, and for any $\HomRing{Y}$--module $M$, $M(W) \cong M[t,t^{-1}]$. On the other hand, if the self-intersection of $[\hat S]$ is nonzero, then the image of $H_2(Y')$ in $W$ agrees with the image of $H_2(Y)$. Therefore, $K(W) \cong H^1(Y) \cong H^1(Y')$, and $M(W) \cong M$ for any $\HomRing{Y}$--module $M$.

If $K$ represents a non-torsion element of $H_1(Y)$, then the map $H_2(Y) \to H_2(W)$ is is an isomorphism, so $K(W) \cong H^1(Y)$. In this case, however, $H^1(Y')$ is smaller than $H^1(Y)$; indeed, we may find an identification of $\HomRing{Y_1}$ with $\HomRing{Y_2}[t,t^{-1}]$.
\end{itemize}

\subsection{The exact triangle with twisted coefficients}

Throughout this section, let $Y$ be a closed, oriented $3$-manifold, and let $\spincs$ be a torsion spin$^c$ structure on $Y$. Let $K \subset Y$ be nulhomologous knot, and let $S$ be a Seifert surface for $K$.

For any integer $m$, let $W_m$ be the $m$-framed $2$-handle cobordism from $Y$ to $Y_m = Y_m(K)$. Let $S_m \subset W_m$ denote the capped-off Seifert surface; when $m=0$, we may view this as lying in $Y_0$. For any integer $k$, let $\spinct_{m,k}$ denote the unique spin$^c$ structure on $W_m$ with
\begin{equation} \label{eq: spinc-extend}
\spinct_{m,k}|_Y = \spincs, \quad \gen{c_1(\spinct_{m,k}), [S_m]} + m = 2k,
\end{equation}
and let $\spincs_{m,k} = \spinct_{m,k}|_{Y_m}$. Additionally, let $W'_m$ denote $W_m$ with reversed orientation, viewed as a cobordism from $Y_m$ to $Y$.

As seen in the previous section, when $m \ne 0$, we have a natural identification $\HomRing{Y} \cong \HomRing{Y_m}$, so we may view any $\HomRing{Y}$--module $M$ as an $\HomRing{Y_m}$--module, and vice versa, and $M(W_m) \cong M$. Likewise, if we consider $M$ as an $\HomRing{Y_m}$ module, then the module induced on $Y$ by $W_m'$ is again isomorphic to $M$. It follows that there are maps
\begin{gather*}
F^\circ_{W_m, \spinct_{m,k}} \co \HF^\circ(Y, \spincs; M) \to \HF^\circ(Y_m,\spincs_{m,k}; M) \\
F^\circ_{W'_m, \spinct_{m,k}} \co \HF^\circ(Y_m, \spincs_{m,k}; M) \to \HF^\circ(Y,\spincs; M).
\end{gather*}

On the other hand, when $m=0$, we have $\HomRing{Y_0} \cong \HomRing{Y}[t^{\pm1}]$, and for any $\HomRing{Y}$--module $M$, $M(W_0) \cong M[t^{\pm1}]$. Hence, for each $k$, we have a map
\[
F^\circ_{W_0, \spincs_{0,k}} \co \HF^\circ(Y, \spincs; M) \to \HF^\circ(Y_0, \spinct_{0,k}; M[t^{\pm1}]),
\]
which can then be extended to a map
\[
\ul F^\circ_{W_0, \spincs_{0,k}} \co \HF^\circ(Y, \spincs; M)[t^{\pm1}] \to \HF^\circ(Y_0, \spinct_{0,k}; M[t^{\pm1}])
\]
by the formula
\[
\ul F^\circ_{W_0, \spincs_{0,k}}(x \otimes t^i) = t^i F^\circ_{W_0, \spincs_{0,k}}(x).
\]

When $m \ne 0$, for each $[k] \in \Z/m$, we define
\[
\HF^\circ(Y_0, [\spincs_{m,k}]; M) = \bigoplus_{l \equiv k \pmod m} \HF^\circ(Y_0, \spincs_{0,l}; M).
\]
A key property is that for $m$ sufficiently large and $\abs{k} \le \frac{m}{2}$, the only nonzero summand in this decomposition of $\HFp(Y_0, [\spincs_{m,k}]; M)$ is $\HFp(Y_0, \spincs_{0,k};M)$. The following theorem is a slight generalization of \cite[Theorem 9.1 and Proposition 9.3]{JabukaMarkProduct}:

\begin{theorem} \label{thm: surgery-triangle}
Let $Y$ be a closed, oriented $3$-manifold, and let $\spincs$ be a torsion spin$^c$ structure on $Y$. Let $K \subset Y$ be a null-homologous knot, and let $S$ be a Seifert surface for $K$. Let $M$ be a module for $\HomRing{Y}$. For any integer $m >0$ and any $[k]\in \Z/m$, there is a sequence of $\HomRing{Y}[t^{\pm1}] \otimes \F[U]$--linear maps:
\begin{equation} \label{eq: surgery-triangle}
\xymatrix@C-0.5in{
\HFp(Y_m(K), \spincs_{m,k};M)[t^{\pm1}] \ar[rr]^{\ul{F}} &  &  \HFp(Y, \spincs; M)[t^{\pm1}] \ar[dl]^(0.4){\ul{G}} \\
& \HFp(Y_0(K), [\spincs_{m,k}]; M[t^{\pm1}]) \ar[ul]^(0.6){\ul{H}} &
}
\end{equation}
Moreover, up to an overall power of $t$, the map $\ul F$ appearing in \eqref{eq: surgery-triangle} is given by
\begin{equation} \label{eq: F-spinc}
\ul F = \sum_{l \equiv k \pmod m} F^+_{W_m', \spinct_{m,l}} \otimes t^{\floor{l/m}}.
\end{equation}
\end{theorem}

\subsection{The mapping cone}

\def\Acomp{\mathbf{A}}
\def\Bcomp{\mathbf{B}}
\def\Ccomp{\mathbf{C}}
\def\Xcomp{\mathbf{X}}

Let $\CFK^\infty(Y,K;\HomRing{Y})$ denote the totally twisted knot Floer complex of $(Y,K)$ with coefficients, coming from a doubly-pointed Heegaard diagram $(\Sigma, \bm\alpha, \bm\beta, w, z)$. This is generated by tuples $[\x, i, j]$ with $j-i = A(\x)$, where $A(\x)$ is the Alexander grading of $\x$, with differential given by
\[
\partial [\x, i, j] = \sum_{\y \in \T_\alpha \cap \T_\beta} \sum_{\substack{\phi \in \pi_2(\x,\y) \\ \mu(\phi)=1}} \#\widehat{\MM}(\Phi) e^{\AA(\phi)} [\y, i-n_w(\phi), j-n_z(\phi)],
\]
where $\AA$ is the additive assignment used for defining twisted coefficients. Define an action of $\F[U,U^{-1}]$ on $C$ by $U \cdot [\x, i, j] = [\x,i-1,j-1]$. Fix an $\HomRing{Y}$--module $M$, and let $C = \CFK^\infty(Y,K; M) = \CFK^\infty(Y,K; \HomRing{Y}) \otimes_{\HomRing{Y}} M$. Let $\Ccomp = C \otimes_{\F[U,U^{-1}]} \F[[U,U^{-1}]$.

Observe that $C$ can be identified as either $\CF^\infty(\Sigma, \bm\alpha, \bm\beta, w; M)$ or $\CF^\infty(\Sigma, \bm\alpha, \bm\beta, z; M)$, by ignoring either $j$ or $i$ respectively. There is a thus a chain homotopy equivalence $\Phi \co C \to C$ which takes $C\{j < s\}$ into $C \{i < s\}$ for any $s$, and therefore descends to a homotopy equivalence $C\{ j \ge s\} \to C\{i \ge s\}$. (Note that there is no control on how $\Phi$ interacts with the second filtration on each complex.) The map $\Phi$ also extends naturally to $\Ccomp$.

For each $s \in \Z$, let $A_s^+ = C\{\max(i, j-s) \ge 0\}$, and let $B^+ = C\{i \ge 0\} = \CFp(Y, \spincs; M)$. Define maps
\[
v_s^\infty, h_s^\infty \co C \to C
\]
as follows: $v_s^\infty$ is the identity, and $h_s^\infty$ is multiplication by $U^s$ followed by $\Phi$. It is easy to verify that these maps descend to
\[
v_s^+, h_s^+ \co A_s^+ \to B^+,
\]
defined just as in \cite{OSzSurgery}. (That is, $v_s^+$ is the projection onto $C\{i \ge 0\}$, and $h_s^+$ is the projection onto $C\{j \ge s\}$, followed by multiplication by $U^s$ to identify this with $C\{j \ge 0\}$, followed by $\Phi$.)

Let
\[
D_{0,s}^\infty \co C[t^{\pm1}] \to C[t^{\pm1}]
\]
be the map of $\HomRing{Y}[t^{\pm1}] \otimes \F[U,U^{-1}]$--modules given by
\[
D_{0,s}^\infty = v_s^\infty + t \cdot h_s^\infty = 1 + tU^s \Phi
\]
This descends to a map
\[
D_{0,s}^+ \co A_s^+[t^{\pm1}] \to B^+[t^{\pm1}],
\]
given by
\[
D_{0,s}^+ = v_s^+ + t \cdot h_s^+.
\]
Let $\ul X_{0,s}^+$ (resp.~$\ul X_{0,s}^\infty)$ denote the mapping cone of $D_{0,s}^+$ (resp.~$D_{0,s}^\infty$). Let $\ul \Xcomp_{0,s}^\infty$ be the $U$-completion of $\ul X_{0,s}^\infty$, which can be viewed as the mapping cone of the extension of $D_{0,s}^\infty$ to $\Ccomp[t^{\pm1}]$. Clearly, $\invlim (\ul X_{0,s}^+, U) = \ul \Xcomp_{0,s}^\infty$.

The surgery formula then states:

\begin{theorem} \label{thm: mapping-cone}
For any $s \in \Z$, there are isomorphisms of relatively graded $\HomRing{Y} \otimes \F[U]$--modules
\begin{gather}
\label{eq: surgery-plus}
H_*(\ul X_{0,s}^+) \cong \HFp(Y_0, \spincs_{0,s}; M[t^{\pm1}]) \\
\label{eq: surgery-infty-comp}
H_*(\ul \Xcomp_{0,s}^\infty) \cong \HFic(Y_0, \spincs_{0,s}; M[t^{\pm1}])
\end{gather}
If $M = M_A$ for some summand $A \subset H^1(Y)$, and $s=0$, we also have
\begin{equation} \label{eq: surgery-infty}
H_*(\ul X_{0,0}^\infty) \cong \HFi(Y_0, \spincs_{0,0}; M[t^{\pm1}]).
\end{equation}
Moreover, under each isomorphism, the map
\[
\ul F^\circ_{W_0, \spinct_{0,s}} \co \HF^{\circ}(Y, \spincs;M )[t^{\pm1}] \to \HF^{\circ}(Y_0, \spincs_{0,s}; M[t^{\pm1}])
\]
is given (up to a power of $t$) by the inclusion of the subcomplex $B^\circ \subset \ul X^\circ_{0,s}$.
\end{theorem}

\begin{proof}
We begin with \eqref{eq: surgery-plus}. Just as in the untwisted case, the large surgery formula, which states that for $m$ sufficiently large and $\abs{s} \le m/2$, there is an identification of $A_s^+$ with $\CFp(Y_m(K), \spincs_{m,s}; M)$, such that the maps $v^+_s$ and $h^+_s$ compute $F^+_{W'_m, \spinct_{m,s}}$ and $F^+_{W'_m, \spinct_{m,s+m}}$, respectively. (See \cite[Theorem 4.4]{OSzKnot} or \cite[Theorem 2.3]{OSzSurgery}; the proof goes through identically with coefficients in $M$.) Just as in \cite{OSzSurgery}, we use the procedure of ``truncation'' applied to the surgery exact sequence of Theorem \ref{thm: surgery-triangle} to obtain \eqref{eq: surgery-plus}. Taking inverse limits and using \eqref{eq: invlim-HFp} yields \eqref{eq: surgery-infty-comp}.

The proof of \eqref{eq: surgery-infty} follows just like in \cite[Lemma 4.10]{LidmanInfinity}, using Lemma \ref{lemma: HFi-free} to observe that $\HFi$ is finitely generated and free over $\F[U,U^{-1}]$.
\end{proof}

We also describe the $H_1$ action. For any $\zeta \in H_1(Y)$, the induced chain map $\AA_\zeta \co C \to C$ commutes with the differential on $C$ and commutes up to homotopy with $\Phi$: say $\AA_\zeta \Phi + \Phi \AA_\zeta = \partial H_\zeta + H_\zeta \partial$. We may then extend $\AA_\zeta$ to $\ul X^\infty_{0,s}$ by the formula
\[
\tilde \AA_\zeta(a,b) = (\AA_\zeta(a), t U^s H_\zeta(a) + A_\zeta(b)),
\]
which descends to $\ul X^+_{0,s}$. These maps give an action of $\Lambda_*(H_1(Y)/Z_M^\perp)$ on $H_*(\ul X^\circ_{0,s})$. Moreover, there is an easy identification
\[
H_1(Y)/Z_M^\perp \cong H_1(Y_0)/Z_{M[t^{\pm1}]}^\perp.
\]
Following through the proof of Theorem \ref{thm: mapping-cone}, it is not hard to see that these chain maps $\tilde \AA_\zeta$ agree with the $H_1$ action on $\HFp(Y_0, \spincs_{0,s}; M[t^{\pm1}])$. (See \cite[Section 4.2]{HeddenNiUnlink}.) We do not need to worry about defining a chain map associated to the homology class of the meridian of $K$ in $H_1(Y_0)$, since its action on $\HFp(Y_0, \spincs_{0,s}; M[t^{\pm1}])$ is $0$.

For the purposes of this paper, the most important consequence of the preceding discussion is the following:
\begin{proposition} \label{prop: HFi-surgery}
Let $Y$ be a closed, oriented $3$-manifold, let $\spincs$ be a torsion spin$^c$ structure on $Y$, and let $M$ be a finitely generated $\HomRing{Y}$--module. Let $K$ be a nulhomologous knot in $Y$, let $W$ be the $2$-handle cobordism from $Y$ to $Y_0(K)$, let $\spinct_0$ be the torsion extension of $\spincs$ to $W$, and let $\spincs_0 = \spinct_0|_{Y_0(K)}$. Then the map
\[
F^\infty_{W,\spinct_0} \co \HF^\infty(Y, \spincs; M) \to \HF^\infty(Y_0(K), \spincs_0; M[t^{\pm1}])
\]
is an isomorphism.
\end{proposition}

\begin{proof}
We apply Theorem \ref{thm: mapping-cone}. Let $C = \CFK^\infty(Y,K;M)$ denote the doubly-filtered knot Floer complex of $(Y,K)$ with coefficients in $M$, and $\Phi \co C \to C$ the homotopy equivalence discussed above. The surgery formula then says that $\HF^\infty(Y_0, \spincs_0; M[t^{\pm 1}])$ can be computed as the mapping cone of
\[
(1+ t\Phi) \co C[t^{\pm1}] \to C[t^{\pm1}],
\]
and the map $F^\infty_{W,\spinct}$ is given (up to a power of $t$) by the inclusion of $C$ into the second copy of $C[t^{\pm1}]$. From this description, it is easy to see that $\HF^\infty(Y_0, \spincs_0; M[t^{\pm1}]) \cong \HF^\infty(Y, \spincs;M)$, where the action of $t$ is given by $\Phi_*^{-1}$, and that $F^\infty_{W,\spinct}$ is an isomorphism.
\end{proof}

\begin{remark} \label{rmk: rationally-nulhom}
Following \cite{OSzRational}, we may adapt the results of this section (specifically Proposition \ref{prop: HFi-surgery}) to the case where $K$ is merely a rationally nulhomologous knot, representing a class of order $d>1$ in $H_1(Y)$. Assume that $K$ has trivial self-linking, so that it has a well-defined $0$-framing. We briefly sketch the necessary modifications to the surgery formula, leaving details to the reader. The set of relative spin$^c$ structures for $K$, $\Spin^c(Y,K)$, forms an affine set for $H^2(Y,K)$. Spin$^c$ structures on $Y_0(K)$ then correspond to the orbits of the action of $\PD[K_\lambda]$, the Poincar\'e dual of the $0$-framed pushoff of $K$, each of which has $d$ elements. The relative spin$^c$ structures also correspond naturally with spin$^c$ structures on the $2$-handle cobordism $W_0(K)$.

Associated to each $\xi \in \Spin^c(Y,K)$, there is a doubly-filtered complex $C_\xi = \CFK^\infty(Y,K,\xi;M)$, and quotients $A^+_\xi$ and $B^+_\xi$ (see \cite{OSzRational} for all definitions). We also have maps
\[
v^\infty_\xi  \co C_\xi \to C_\xi, \quad h^\infty_\xi \co C_\xi \to C_{\xi + \PD[K_\lambda]},
\]
which induce
\[
v^+_\xi \co A^+_\xi \to B^+_\xi, \quad h^\infty_\xi \co A^+_\xi \to B^+_{\xi + \PD[K_\lambda]},
\]
defined similar to the above. Specifically, $v^\infty_\xi$ is the identity, and $h^\infty_\xi$ is a homotopy equivalence induced by Heegaard moves.

Suppose $\{\xi_1, \dots, \xi_d\}$ is the orbit corresponding to a torsion spin$^c$ structure $\spincs_0$ on $Y_0$,  where $\xi_{i+1} = \xi + \PD[K_\lambda]$ (indices modulo $d$). Let $\spincs_i$ be the (absolute) spin$^c$ structure on $Y$ extending $\xi_i$, and $\spinct_i$ the spin$^c$ structure on $W_0$ corresponding to $\xi_i$. Write $C_i$ for $C_{\xi_i}$ and $\Phi_i$ for $h^\infty_{\xi_i}$. The twisted mapping cone that computes $\HF^\infty(Y_0,\spincs_0; M[t^{\pm1}])$ has the form
\[
\xymatrix{
C_1[t^{\pm1}] \ar[d]^{1} \ar[dr]^{\Phi_1} & C_2[t^{\pm1}] \ar[d]^{1} \ar[dr]^{\Phi_2} & \dots \ar[dr]^{\Phi_{d-1}} & C_d[t^{\pm1}] \ar[d]^{1}
\ar `r[d] `[dd] ^{t \cdot \Phi_d} `l[dlll] [dlll]
\\
C_1[t^{\pm1}] & C_2[t^{\pm1}] & \dots & C_d[t^{\pm1}] & \\
& & & & }
\]
Up to isomorphism, it doesn't matter which of the $\Phi_i$ arrows comes with a power of $t$; the important point is that exactly one of them does. The map
\[
F^\infty_{W_0,\spinct_i} \co  \HFi(Y,\spincs_i;M) \to \HFi(Y_0,\spincs_0; M[t^{\pm1}])
\]
is given (up to a power of $t$) by the inclusion of $C_i$ in the bottom row. Just as in the proof of Proposition \ref{prop: HFi-surgery}, we deduce that this map is an isomorphism. We will make use of this generalization in Section \ref{sec: cross-section}.
\end{remark}

\section{Twisted correction terms} \label{sec: twisted-d}
In this section, we define the twisted correction terms. Throughout, let $Y$ be a closed, oriented $3$-manifold, and let $A \subset H^1(Y)$ be a direct summand on which the triple cup product vanishes. As above, let $\HomRing{Y} = \F[H^1(Y)]$, and let $M_A = \F[H^1(Y)/A]$, viewed as an $\HomRing{Y}$--module.

\subsection{Construction of the invariants}

To begin, we show that $\HFi(Y,\spincs;M_A)$ is standard, in the following sense.

\begin{theorem} \label{thm: HFi-standard}
Let $Y$ be a closed, oriented $3$-manifold, let $\spincs$ be a torsion spin$^c$ structure on $Y$. Let $A \subset H^1(Y)$ be a direct summand on which the triple cup product vanishes, and let $M_A = \F[H^1(Y)/A]$. Then
\[
\HF^\infty(Y, \spincs; M_A) \cong \Lambda^*(A) \otimes \F[U,U^{-1}]
\]
as a $\Lambda^*(H_1(Y)/A^\perp) \otimes \F[U,U^{-1}]$--module (where the action of $\Lambda^*(H_1(Y)/A^\perp)$ on $\Lambda^*(A)$ is induced from the natural action of $\Lambda^*(H_1(Y)/\Tors)$ on $\Lambda^*(H^1(Y))$).
\end{theorem}

\begin{proof}
We induct on the rank of $H^1(Y)/A$, starting with the extremal case when $A = H^1(Y)$ and the triple cup product on $H^1(Y)$ vanishes identically. The statement in this case follows from \cite{LidmanInfinity}, as explained in \cite[Theorem 3.2]{LevineRubermanCorrection}.

For the induction, assume that $H^1(Y)/A \ne 0$. Let $J \subset Y$ be a knot representing a primitive homology class in $A^\perp \subset H_1(Y)$ such that $\gen{\beta, [J]} = 1$ for some $\beta \in H^1(Y) \minus A$. Let $Z$ be obtained by surgery on $J$ with some arbitrary framing, and let $K \subset Z$ denote the core of the surgery solid torus, so that $Y = Z_0(K)$. Let $W$ be the $2$-handle cobordism from $Z$ to $Y$, and $\iota_Y \co Y \to W$ and $\iota_Z \co Z \to W$ the inclusions.

The map $\iota_Z^* \co H^1(W) \to H^1(Z)$ is an isomorphism, and $\iota_Y^* \circ (\iota_Z^*)^{-1}$ restricts to an injection on $A$. Let $A' \subset H^1(Z)$ be the image of this restriction, and let $M_{A'} = \F[H^1(Z)/A']$. Then $M_A \cong M_{A'}(W) \cong M_{A'}[t^{\pm1}]$. Also, let $\spincs'$ be the restriction to $Z$ of the unique spin$^c$ structure on $W$ that extends $\spincs$.

By the induction hypothesis,
\[
\HF^\infty(Z, \spincs'; M_{A'}) \cong \Lambda^*(A') \otimes \F[U,U^{-1}]
\]
as a $\Lambda^*(H_1(Z)/A^{\prime \perp}) \otimes \F[U,U^{-1}]$--module. The result then follows from Proposition \ref{prop: HFi-surgery}.
\end{proof}

\begin{remark} \label{rmk: totally-twisted}
As noted in the proof of Lemma \ref{lemma: HFi-free} above, the other extremal case of Theorem \ref{thm: HFi-standard} was proven by Ozsv\'ath and Szab\'o \cite[Theorem 10.12] {OSzProperties}: when $A=0$, the totally twisted homology satisfies
\[
\HF^\infty(Y, \spincs; \HomRing{Y}) \cong \F[U,U^{-1}].
\]
\end{remark}

\begin{remark} \label{rmk: module-structure}
Note that Theorem \ref{thm: HFi-standard} does \emph{not} describe the structure of $\HFi(Y,\spincs;M_A)$ as an $\HomRing{Y}$--module. The action of any element of $H^1(Y)$ is a grading-preserving automorphism of $\HFi(Y, \spincs;M_A)$ that commutes with the action of $\Lambda^*(H_1(Z)/A^\perp) \otimes \F[U,U^{-1}]$, but in principle this map need not be the identity.
\end{remark}

We may now define the $d$ invariant that we use below, which is analogous to the $\dtop$ invariant defined in \cite[Definition 3.3]{LevineRubermanStrleNonorientable}. We make use of notation from \cite{LevineRubermanCorrection}. First, given any finitely generated, free abelian group $V$ and any $\Lambda^*(V)$--module $N$, define $\QQ^V (N) = N / (V \cdot N)$ and $\KK^V (N) = \{n \in N \mid v \cdot n = 0 \ \forall v \in V\}$ (i.e., the quotient and kernel of the action of $V$, respectively). We sometimes omit the superscripts if they are understood from context.

\begin{definition} \label{def: d}
Let $Y$ be a closed, oriented $3$-manifold, $\spincs$ a torsion spin$^c$ structure, and $A \subset H^1(Y)$ a subspace on which the triple cup product vanishes. Let
\[
\pi\co \HFi(Y,\spincs;M_A) \to \HFp(Y,\spincs;M_A)
\]
denote the canonical map. Then there are isomorphisms
\begin{align}
\label{eq: QHFi} \QQ^{H_1(Y)/A^\perp} (\HFi(Y,\spincs;M_A)) &\cong \F[U,U^{-1}] \\
\label{eq: Qim(pi)} \QQ^{H_1(Y)/A^\perp} (\im(\pi))  &\cong \F[U,U^{-1}] / U \F[U],
\end{align}
such that the induced map
\[
\bar\pi \co \QQ^{H_1(Y)/A^\perp} (\HFi(Y,\spincs;M_A))  \to \QQ^{H_1(Y)/A^\perp} (\im(\pi) )
\]
is the natural projection. The \emph{correction term} $\dtw(Y,\spincs; M_A) \in \Q$ is defined as minimal grading in which $\bar \pi$ is nontrivial, or equivalently as the grading of $1 \in \F[U,U^{-1}] / U \F[U]$ under the identification \eqref{eq: Qim(pi)}. The \emph{shifted correction term} is defined as
\begin{equation} \label{eq: d-shifted}
\dsh(Y,\spincs;M_A) = \dtw(Y, \spincs; M_A) - \rank(A) + \frac{b_1(Y)}{2}.
\end{equation}
If $H^2(Y)$ is torsion-free, so that $Y$ has a unique torsion spin$^c$ structure, we sometimes omit $\spincs$ from the notation.
\end{definition}

When $A = H^1(Y)$ (so that $M_A = \F$) and the triple cup product vanishes identically, $\dtw(Y, \spincs; \F) = \dtop(Y, \spincs)$. When $A=0$ (so that $M_A = \HomRing{Y}$), $\dtw(Y, \spincs;\HomRing{Y})$ is precisely the invariant $\ul d$ defined by Behrens and Golla \cite{BehrensGollaCorrection}.

\begin{example} \label{ex: S1xS2}
When $Y = S^1 \times S^2$, it is easy to compute directly from a Heegaard diagram that
\begin{align*}
\dtw(Y; \F) &= \frac12 & \dsh(Y; \F) &= 0\\
\dtw(Y; \HomRing{Y}) &= -\frac12 & \dsh(Y; \HomRing{Y}) &= 0.
\end{align*}
\end{example}

\begin{example} \label{ex: trefoil}
If $Y$ is obtained by $0$-surgery on a knot $K \subset S^3$, then $\dsh(Y;\F)$ and $\dsh(Y; \HomRing{Y})$ are both determined by the knot Floer complex of $K$. Specifically, combining \cite[Example 3.9]{BehrensGollaCorrection}, \cite[Proposition 4.12]{OSzAbsolute}, and \cite[Proposition 1.6]{NiWuCosmetic}, we have:
\begin{gather}
\label{eq: 0surgery-untwisted}
\dsh(Y; \F) = \dtw(Y;\F) - \frac12 = \dtop(Y) -\frac12 = d(S^3_1(K)) = -2 V_0(K) \\
\label{eq: 0surgery-twisted}
\dsh(Y; \HomRing{Y}) = \dtw(Y; \HomRing{Y}) + \frac12  = \dbot(Y)+\frac12 = d(S^3_{-1}(K)) = 2 V_0(\bar K),
\end{gather}
where $V_0$ is a nonnegative integer invariant defined by Ni and Wu \cite[Section 2.2]{NiWuCosmetic}, and $\bar K$ denotes the mirror of $K$. (The second inequality in \eqref{eq: 0surgery-twisted}, proven by Behrens and Golla, is special to the case of $0$-surgery on knots in $S^3$.)

In particular, if $K$ is either the right-handed trefoil $T_{2,3}$ or its positive, untwisted Whitehead double $D(T_{2,3})$, then $V_0(K)=1$ and $V_0(\bar K) = 0$. The statement for $T_{2,3}$ is a straightforward computation; the statement for $D(T_{2,3})$ follows from the fact that $\CFK^\infty(D(T_{2,3}))$ is isomorphic to $\CFK^\infty(T_{2,3})$ plus an acyclic summand that does not affect $V_0$ \cite[Proposition 6.1]{HeddenKimLivingston}. Hence, if $Y$ is obtained by $0$-surgery on either of these knots, we have:
\begin{align*}
\dtw(Y; \F) &=  -\frac32 & \dsh(Y; \F) &= -2\\
\dtw(Y; \HomRing{Y}) &=  -\frac12 & \dsh(Y; \HomRing{Y}) &= 0 \\
\dtw(-Y; \F) &= \frac12 & \dsh(-Y; \F) &= 0 \\
\dtw(-Y; \HomRing{-Y}) &= \frac32 & \dsh(-Y; \HomRing{-Y}) &= 2.
\end{align*}
(The results for $0$-surgery on the trefoil were also proven earlier by Ozsv\'ath and Szab\'o \cite{OSzAbsolute}.)
\end{example}

\begin{example} \label{ex: T3}
Let $T^3$ denote the $3$-torus. Because the triple cup product on $H^1(T^3)$ is nonvanishing, the invariant $\dtw(T^3; M_A)$ is only defined when $\rank A = 0$, $1$, or $2$. When $A=0$, \cite[Proposition 8.5]{OSzAbsolute} shows that
\begin{align*}
\dtw(T^3; \HomRing{T^3}) &=  \frac12 & \dsh(T^3; \HomRing{T^3}) &= 2.
\end{align*}
On the other hand, we will see below in Example \ref{ex: T3-partial} that when $\rank A = 1$ or $2$, $\dsh(T^3; M_A)=0$. Since any automorphism of $H^1(T^3)$ can be realized by a self-diffeomorphism of $T^3$, it suffices to compute these invariants for a single subspace $A$ of either rank. Note also that $T^3$ admits orientation-reversing diffeomorphisms, so the the same statements hold with either orientation on $T^3$.
\end{example}

\subsection{Relation with untwisted invariants}

We now describe the relationship between Definition \ref{def: d} and the invariants defined in \cite{LevineRubermanCorrection}. Suppose the triple cup product on $H^1(Y)$ vanishes identically, so that the untwisted homology group $\HFi(Y,\spincs;\F)$ is standard:
\[
\HFi(Y,\spincs; \F) \cong \Lambda^* H^1(Y) \otimes \F[U,U^{-1}]
\]
as a $\Lambda^*(H_1(Y)/\operatorname{Tors}) \otimes \F[U,U^{-1}]$--module. In \cite{LevineRubermanCorrection}, we defined an ``intermediate correction term'' $d(Y, \spincs, V)$ associated to each subspace $V \subset H_1(Y)$. In particular, $\dtop(Y,\spincs) = d(Y,\spincs,\{0\})$ and $\dbot(Y,\spincs) = d(Y,\spincs,H_1(Y))$. The two constructions are related as follows:

\begin{proposition} \label{prop: d-standard}
If the triple cup product on $H^1(Y)$ vanishes identically, then for each summand $A \subset H^1(Y)$, we have
\begin{equation} \label{eq: d-standard}
\dtw(Y, \spincs; M_A) \le d(Y, \spincs, A^{\perp}), 
\end{equation}
and therefore
\begin{align}
\label{eq: d-standard-dbot}
\dtw(Y,\spincs; M_A) &\le \dbot(Y,\spincs) + \rank(A)  \\
\label{eq: dsh-standard}
\dsh(Y,\spincs; M_A) &\le \dbot(Y,\spincs) + \frac{b_1(Y)}{2}.
\end{align}
\end{proposition}

(The case where $A=0$ was proven by Behrens and Golla \cite[Proposition 3.8]{BehrensGollaCorrection}.)

\begin{proof}
To begin, note that $M_A = \F[H^1(Y)/A]$ is a commutative ring with unit, not just a module over $\HomRing{Y}$, and the projection $\HomRing{Y} \to M_A$ is a ring homomorphism. Let $n = b_1(Y)$ and $k = \rank A$; we assume $n>0$. For concreteness, let $a_1, \dots, a_n$ be a basis for $H^1(Y)$ such that $a_1, \dots, a_k$ are a basis for $A$. Let $\zeta_1, \dots, \zeta_n$ be the dual basis for $H_1(Y)/\Tors$, so that $A^\perp = \Span(\zeta_{k+1}, \dots, \zeta_n)$.

Let $C_* = \CF^\infty(Y,\spincs;M_A)$, with differential denoted by $\partial$. As a simplification, let us shift the homological grading on $C_*$ so that it lies in $\Z$ (rather than $\Z+q$ for some rational number $q$). Furthermore, if $A=0$,  so that $\HFi(Y,\spincs;M_A) \cong \F[U,U^{-1}]$, we assume that the nonzero groups are in even grading. If we consider $\F$ as an $M_A$--module, where each element of $H^1(Y)/A$ acts as the identity, then by definition, $H_q(C_*) = \HFi_q(Y,\spincs; M_A)$, while $H_q(C_* \otimes_{M_A} \F) = \HFi_q(Y,\spincs)$.

Since the untwisted $\HFi(Y,\spincs)$ is standard, we have
\[
H_q(C \otimes_{M_A} \F) \cong \F^{2^{n-1}}.
\]
By Theorem \ref{thm: HFi-standard}, if $k=0$, we have
\begin{equation} \label{eq: Hq(C)k=0}
H_q(C_*) \cong
\begin{cases}
  \F & q \text{ even} \\
  0 & q \text{ odd},
\end{cases}
\end{equation}
while if $k>0$, we have
\begin{equation} \label{eq: Hq(C)k>0}
H_q(C_*) \cong \F^{2^{k-1}}
\end{equation}
for all $q$. As in Remark \ref{rmk: module-structure}, right now we only know that the isomorphisms \eqref{eq: Hq(C)k=0} and \eqref{eq: Hq(C)k>0} hold on the level of groups; we shall see shortly that they hold on the level of $M_A$--modules as well.

Consider the following commutative diagram:
\begin{equation} \label{eq: coeff-change}
\xymatrix@C=0.75in{
\HFi(Y,\spincs;M_A) \ar[r]^{\pi_{M_A}} \ar[d]_{\otimes 1} & \HFp(Y,\spincs; M_A) \ar[d]_{\otimes 1} \\
\HFi(Y,\spincs;M_A) \otimes_{M_A} \F \ar[r]^{\pi_{M_A} \otimes 1} \ar[d]_{g^\infty} & \HFp(Y,\spincs; M_A) \otimes_{M_A} \F \ar[d]_{g^+} \\
\HFi(Y,\spincs;\F) \ar[r]^{\pi_\F} & \HFp(Y,\spincs; \F) }
\end{equation}
Here $\pi_{M_A}$ and $\pi_\F$ are the usual maps $\HFi \to \HFp$, and $g^\infty$ and $g^+$ are the natural change-of-coefficient maps. As above, let $\KK^{A^\perp} \HFi(Y,\spincs;\F)$ denote the subspace of $\HFi(Y,\spincs;\F)$ consisting of all $x \in \HFi(Y, \spincs;\F)$ for which $\zeta \cdot x = 0$ for all $\zeta \in A^\perp$, and let $J^+(Y,\spincs,A^\perp)$ denote the image of the restriction of $\pi_\F$ to $\KK^{A^\perp} \HFi(Y,\spincs;\F)$. The invariant $d(Y,\spincs, A^\perp)$ is defined to be the minimal grading in which the induced map
\[
\bar \pi_\F \co \QQ^{H_1(Y)/A^\perp} \KK^{A^\perp} \HFi(Y,\spincs;\F) \to \QQ^{H_1(Y)/A^\perp}  J^+(Y,\spincs,A^\perp)
\]
is nontrivial.

\begin{claim} \label{claim: MA-action}
The upper-left vertical map in \eqref{eq: coeff-change} is an isomorphism; equivalently, the action of $M_A$ on $\HFi(Y,\spincs;M_A)$ is trivial.
\end{claim}

\begin{claim} \label{claim: g-infty}
The map $g^\infty$ is injective with image equal to $\KK^{A^\perp} \HFi(Y,\spincs;\F)$.
\end{claim}

Assuming these two claims, it follows that we have a commutative diagram
\[
\xymatrix{
\HFi(Y,\spincs;M_A) \ar[r]^{\pi_{M_A}} \ar[d]_{\cong} & \HFp(Y,\spincs; M_A) \ar[d] \\
\KK^{A^\perp} \HFi(Y,\spincs;\F) \ar[r]^-{\pi_\F} & J^+(Y,\spincs, A^\perp)
}
\]
which then descends to
\[
\xymatrix{
\QQ \HFi(Y,\spincs;M_A) \ar[r]^{\pi_{M_A}} \ar[d]_{\cong} & \QQ \HFp(Y,\spincs; M_A) \ar[d] \\
\QQ \KK^{A^\perp} \HFi(Y,\spincs;\F) \ar[r]^-{\pi_\F} & \QQ J^+(Y,\spincs, A^\perp).
}
\]
By comparing the definitions of the two $d$ invariants, it is then easy to see that \eqref{eq: d-standard} holds. \eqref{eq: d-standard-dbot} and \eqref{eq: dsh-standard} then follow from \cite[Proposition 3.4]{LevineRubermanCorrection}.

To prove Claims \ref{claim: MA-action} and \ref{claim: g-infty}, we use the universal coefficients spectral sequence, which we explain in some detail because morphisms of spectral sequences can be confusing. To begin, take a free resolution of $\F$ as an $M_A$--module:
\[
0 \xleftarrow{} \F \xleftarrow{d_0} F_0 \xleftarrow{d_1} F_1 \xleftarrow{d_2} \dots \xleftarrow{d_{n-k}} F_{n-k} \xleftarrow{} 0,
\]
where $F_p = M_A^{\binom{n-k}{p}}$. Consider the complex
\[
\CC_s = \bigoplus_{p+q=s} C_q \otimes F_p
\]
(nonzero only when $0 \le p \le n-k$) with differential $D_s \co \CC_s \to \CC_{s-1}$ given by
\[
D_s = \sum_{p+q=s} (-1)^s \partial_q \otimes d_p.
\]
Observe that $H_s(\CC_*) \cong H_s(C_* \otimes_{M_A} \F) \cong \HFi_s(Y,\spincs; \F)$.

The spectral sequence comes from considering the $p$ filtration on $\CC_*$, so that the differential on the $E^r$ has $(p,q)$--bigrading $(-r,r-1)$. The $E^1$ page is given by
\begin{equation} \label{eq: E1}
E^1_{p,q} \cong H_q(C_*) \otimes_{M_A} F_p \cong H_q(C_*)^{\binom{n-k}{p}}
\end{equation}
and the $E^2$ page is given by
\begin{equation} \label{eq: E2}
E^2_{p,q} \cong \Tor^{M_A}_p(H_q(C_*), \F).
\end{equation}
In particular, in the $p=0$ column, we have
\begin{equation} \label{eq: E20}
E^1_{0,*} \cong \HFi(Y,\spincs;M_A) \quad \text{and} \quad E^2_{0,*} \cong \HFi(Y,\spincs;M_A) \otimes_{M_A} \F,
\end{equation}
and the upper-left vertical map in \eqref{eq: coeff-change} is the natural quotient map. Furthermore, there is a filtration
\begin{equation} \label{eq: homology-filt}
0 = G_s^{-1} \subset G_s^0 \subset G_s^1 \subset \cdots \subset G_s^{n-k} = H_s(\CC_*)
\end{equation}
so that
\begin{equation} \label{eq: Einfty}
E^\infty_{p,q} \cong G_{p+q}^p/G_{p+q}^{p-1};
\end{equation}
in particular, the $p=0$ column $E^\infty_{0,*}$ is identified with the subspace $G_*^0$. The map $g^\infty$ from \eqref{eq: coeff-change} is given by the identification \eqref{eq: E20}, followed by the successive quotients taking $E^2_{0,*} \to E^\infty_{0,*}$, followed by the inclusion of $G_*^0$ into $H_*(\CC_*) \cong \HFi(Y,\spincs; \F)$.

By \eqref{eq: E1}, we have
\begin{equation} \label{eq: E1pq}
\dim_\F E^1_{p,q} = \binom{n-k}{p} \dim_\F H_q(C_*) =
\begin{cases}
\binom{n-k}{p} & k=0, \, q \text{ even} \\
0 & k=0, \, q \text{ odd} \\
\binom{n-k}{p} 2^{k-1} & k>0.
\end{cases}
\end{equation}
Summing over $p+q=s$, we see that
\[
\sum_{p+q=s} \dim_\F E^1_{p,q} = 2^{n-1} = \dim_\F H_s(\CC_*) = \sum_{p+q=s} \dim_\F E^\infty_{p,q},
\]
which implies that the spectral sequence collapses at the $E^1$ page. Looking in the $p=0$ column, we see that the successive quotient maps
\[
E^1_{0,q} \to E^2_{0,q} \to \dots \to E^\infty_{0,q}
\]
are all isomorphisms, which proves Claim \ref{claim: MA-action} and the injectivity statement of Claim \ref{claim: g-infty}.

It remains to identify $G^0_*$ with $\KK^{A^\perp} \HFi(Y,\spincs;M)$. For each $i=1, \dots, n$, there is a chain map $\AA_{\zeta_i} \co C_* \to C_{*-1}$; these give rise to the action of $H_1$. As noted above, the maps $\AA_{\zeta_{k+1}}, \dots, \AA_{\zeta_n}$ are null-homotopic (see \cite[Remark 5.2]{JabukaMarkProduct}), but they are still defined on the chain level. Indeed, we extend each $A_{\zeta_i}$ to a map on $\CC_*$ by tensoring with the identity map on $F_*$. The maps $\AA_{\zeta_i *}$ induced on the homology of $\CC_*$ (which, as noted above, is isomorphic to $\HFi(Y,\spincs;\F)$) then generate the action of $\Lambda^*(H_1(Y)/\Tors)$ on $\HFi(Y,\spincs;\F)$; that is, $\zeta_i \cdot x = \AA_{\zeta_i *} (x)$.

Moreover, the restriction of $A_{\zeta_i *}$ to $G^0_*$ agrees with the action of $\zeta_i$ on $\HFi(Y, \spincs; M_A)$. For $i=k+1, \dots, n$, this action vanishes, so
\[
G_q^0 \subset \KK^{A^\perp}_q \HFi(Y,\spincs;\F)
\]
for each grading $q$. Because $\HFi(Y,\spincs;\F)$ is standard, we can see that $\KK^{A^\perp} \HFi(Y,\spincs;\F)$ is a free $\F[U,U^{-1}]$ module with
\[
\rank_{\F[U,U^{-1}]} \KK^{A^\perp} \HFi(Y,\spincs;\F) = \frac{1}{2^{n-k}} \rank_{\F[U,U^{-1}]} \HFi(Y,\spincs;\F) = 2^k.
\]
(Taking the kernel of each $\zeta_i$, for $i=k+1, \dots, n$ cuts down the rank by a factor of $2$.) If $k=0$, then $\KK^{A^\perp} \HFi(Y,\spincs; \F)$ is a single tower, with $0$ and $\F$ in alternating gradings; otherwise, $\KK^{A^\perp} \HFi(Y,\spincs; \F)$ has dimension $2^{k-1}$ in each grading. By \eqref{eq: E1pq}, we see that
\[
\dim_\F \KK^{A^\perp}_q \HFi(Y,\spincs;\F) = \dim_\F E^1_{0,q} = \dim_\F G^0_q,
\]
so $G^0_* = \KK^{A^\perp} \HFi(Y,\spincs;M)$ as required.
\end{proof}

\begin{figure}
\[
\xymatrix@R=0.2in@C=0.2in{
9/2 && U^{-2}a \ar@{-->}[dr]^{1-t}  \ar@{.>}[d]_{1-t^2} &
   &&& U^{-2}a \ar@{-->}[dr]^{1-t}  \ar@{.>}[d]_{1-t^2} & \\
7/2 && U^{-2}b \ar[dr] & U^{-1}c \ar@{-->}[d]^{1-t}
   &&& U^{-2}b \ar[dr] & U^{-1}c \ar@{-->}[d]^{1-t} \\
5/2 && U^{-1}a \ar@{-->}[dr]  \ar@{.>}[d] & U^{-1}d
   &&& U^{-1}a \ar@{-->}[dr]  \ar@{.>}[d] & U^{-1}d \\
3/2 && U^{-1}b \ar[dr] & c \ar@{-->}[d]
   &&& U^{-1}b \ar[dr] & c \ar@{-->}[d]\\
1/2 && a \ar@{-->}[dr]  \ar@{.>}[d] & d
   &&& a \ar@{.>}[d] & d \\
-1/2 && b \ar[dr] & U c \ar@{-->}[d]
    &&& b  &  \\
-3/2 && U a \ar[dr] \ar@{.>}[d] & U d
    &&  &  &  \\
-5/2 && U b \ar[dr] & U^2 c \ar@{-->}[d]
    &&  &  &  \\
-7/2 && & U^2 d
    && & & \\
}
\]
\caption{$\CFi$ (left) and $\CFp$ (right) for a hypothetical manifold $(Y,\spincs)$ with $\dtw(Y,\spincs; \HomRing{Y}) < \dbot(Y,\spincs)$, as in Example \ref{ex: not-equal}.} \label{fig: not-equal}
\end{figure}

\begin{example} \label{ex: not-equal}
Although we do not know of an actual manifold $Y$ for which equality fails to hold in \eqref{eq: d-standard}, this seems unlikely to be true in general. Figure \ref{fig: not-equal} represents the totally twisted complexes $\CFi(Y,\spincs;\HomRing{Y})$ and $\CFp(Y,\spincs;\HomRing{Y})$ for a hypothetical $(Y,\spincs)$ with $b_1(Y)=1$. Writing $\HomRing{Y} = \F[t^{\pm1}]$, we view $\CFi(Y,\spincs;\HomRing{Y})$ as a complex over $\F[t^{\pm1}, U^{\pm 1}]$ generated by $a, b, c, d$. A solid arrow represents a coefficient of $1$ in the differential, a dashed arrow represents $1-t$, and a dotted arrow represents $1-t^2 = (1-t)^2$. The pattern repeats infinitely in both directions in $\CFi$ and infinitely upward in $\CFp$. The numbers at the left represent the Maslov gradings, which we have chosen in analogy with $S^1 \times S^2$.

Clearly, $\HFi(Y,\spincs;\HomRing{Y})$ is isomorphic to $\F[U,U^{-1}]$, generated as a $\F[t^{\pm1}, U^{\pm1}]$--module by $(1-t)b + Uc$, with the relation $(1-t)((1-t)b+Uc)=0$. Also, $\HFp(Y,\spincs;\HomRing{Y})$ is generated as a $\F[t^{\pm1}]$--module by $\{U^n ((1-t)b+Uc) \mid n \le -1\}$ along with $b$, with the relations that $(1-t) U^n((1-t)b+Uc)=0$ and $(1-t)^2 b=0$. We thus see that $\dtw(Y,\spincs; \HomRing{Y}) = -1/2$. (Notice that the short exact sequence
\[
0 \to \im(\pi) \to \HFp(Y,\spincs;\HomRing{Y}) \to \HFp_{\text{red}}(Y,\spincs; \HomRing{Y}) \to 0
\]
does not split over $\HomRing{Y}$, although it does split over $\F$.)

On the other hand, we can also view the same figure as representing the untwisted complexes $\CFi(Y,\spincs;\F)$ and $\CFp(Y,\spincs;\F)$. Now the solid arrows represent the differential, the dashed arrows represent the chain map $\AA_\zeta$ associated to a generator of $H_1(Y)$, and the dotted arrows represent a chain null-homotopy of $\AA_\zeta^2$. Here, $\HFi(Y,\spincs;\F)$ is generated over $\F[U,U^{-1}]$ by $a$ and $c$, with $c$ generating the ``bottom tower'' $\KK^{\HomRing{Y}} \HFi(Y,\spincs;\F)$. Also, $\HFp(Y,\spincs;\F)$ is generated as a $\F$--module by
$\{U^n a, U^n c \mid n \le 0\} \cup \{b\}$. We therefore deduce that $\dbot(Y,\spincs)=3/2$.

Moreover, it is not hard to modify the construction to make the difference $\dbot(Y,\spincs) - \dtw(Y,\spincs;\HomRing{Y})$ arbitrarily large.

As noted above in Example \ref{ex: trefoil}, Behrens and Golla \cite[Example 3.9]{BehrensGollaCorrection} proved that for any knot $K \subset S^3$, $\ul d(S^3_0(K)) = \dbot(S^3_0(K))$. Thus, a manifold for which \eqref{eq: d-standard} is a strict inequality, as in the putative example just discussed, would not be homology cobordant to $0$-surgery on any knot in $S^3$.
\end{example}

\begin{remark} \label{rmk: neg-def}
Proposition \ref{prop: d-standard} implies that the twisted $d$ invariants can in principle give stronger constraints on intersection forms of $4$-manifolds bounded by $Y$ than the untwisted $d$ invariants from \cite{LevineRubermanCorrection}. For instance, if $Z$ is a negative semi-definite $4$-manifold bounded by $Y$ such that the restriction map $H^1(Z) \to H^1(Y)$ is trivial, and $\spinct$ is a spin$^c$ structure on $Z$ whose restriction to $Y$ is torsion, Ozsv\'ath and Szab\'o \cite[Theorem 9.15]{OSzAbsolute} showed that
\[
c_1(\spinct)^2 + b_2^-(Z) \le 4 \dbot(Y,\spinct|_{Y}) + 2b_1(Y),
\]
and Behrens and Golla \cite[Theorem 1.1]{BehrensGollaCorrection} proved an analogous statement with $d(Y,\spinct|_Y; \HomRing{Y})$ in place of $\dbot(Y,\spinct|_Y)$. The latter result is thus a potentially stronger bound. Likewise, for any summand $A \subset H^1(Y)$, it is not hard to prove a stronger analogue of \cite[Theorem 4.7]{LevineRubermanCorrection} using $\dtw(Y,\spincs;M_A)$ in place of $d(Y,\spincs,A^\perp)$ (where $A$ is chosen such that $A^\perp=V$).
\end{remark}

\begin{definition} \label{def: d-simple}
For a $3$-manifold $Y$ and a torsion spin$^c$ structure $\spincs$ on $Y$, we say that $(Y,\spincs)$ is \emph{$d$-simple} if the triple cup product on $H^1(Y)$ vanishes identically (so that $\HFi(Y,\spincs;\F)$ is standard) and for every subspace $A \subset H^1(Y)$, equality holds in \eqref{eq: dsh-standard}, i.e.
\begin{equation} \label{eq: dsh-dbot}
\dsh(Y,\spincs;A) = \dbot(Y,\spincs) + \frac{b_1(Y)}{2}.
\end{equation}
We say that $Y$ is $d$-simple if $(Y,\spincs)$ is $d$-simple for each torsion spin$^c$ structure $\spincs$ on $Y$.
\end{definition}

If $(Y,\spincs)$ is $d$-simple, then for each $A \subset H^1(Y)$, \eqref{eq: d-standard} is an equality, meaning that
\[
d(Y,\spincs, A^\perp) = \dbot(Y,\spincs) + \rank(A).
\]
In other words, the untwisted $d$ invariants of $(Y,\spincs)$ are simple in the sense of \cite[Corollary 3.5]{LevineRubermanCorrection}. In particular, we have
\[
\dtop(Y,\spincs) = \dbot(Y,\spincs) + b_1(Y),
\]
and hence
\begin{equation} \label{eq: dsh-dtop}
\dsh(Y,\spincs;M_A) = \dtop(Y,\spincs) - \frac{b_1(Y)}{2}.
\end{equation}

\subsection{Orientation reversal} \label{subsec: d-reversal}

Unlike with the original $d$ invariant for rational homology spheres, Examples \ref{ex: S1xS2} and \ref{ex: trefoil} show that $\dsh(Y,\spincs; M_A)$ does not determine $\dsh(-Y,\spincs;M_A)$. The only relation between these quantities occurs in the extremal cases where $A=0$ or $A=H^1(Y)$:

\begin{proposition} \label{prop: reversal}
Let $Y$ be a closed, oriented $3$-manifold and let $\spincs$ be a torsion spin$^c$ structure on $Y$. Then
\begin{equation} \label{eq: reversal-totally}
\dsh(Y,\spincs; \HomRing{Y}) + \dsh(-Y, \spincs; \HomRing{-Y}) \ge 0.
\end{equation}
If the triple cup product on $H^1(Y)$ vanishes, so that $\HFi(Y,\spincs;\F)$ is standard, then
\begin{equation} \label{eq: reversal-untwisted}
\dsh(Y,\spincs; \F) + \dsh(-Y, \spincs; \F) \le 0.
\end{equation}
\end{proposition}

\begin{proof}
Behrens and Golla \cite[Proposition 3.7]{BehrensGollaCorrection} showed that the totally twisted $d$ invariant is additive under connected sums; thus,
\[
\dtw(Y,\spincs; \HomRing{Y}) + \dtw(-Y, \spincs; \HomRing{-Y}) = \dtw(Y \conn {-Y}, \spincs \conn \spincs; \HomRing{Y \conn {-Y}}).
\]
Note that $Y \conn {-Y}$ is the boundary of the four-manifold $Z = (Y \minus B^3) \times [0,1]$, whose intersection form vanishes identically. Applying \cite[Theorem 1.1]{BehrensGollaCorrection} (see Remark \ref{rmk: neg-def} above), we deduce that
\[
0 \le \dtw(Y \conn {-Y}, \spincs \conn \spincs; \HomRing{Y \conn {-Y}}) + b_1(Y).
\]
This implies \eqref{eq: reversal-totally}.

For the second statement, we have $\dtw(Y,\spincs; \F) = \dtop(Y,\spincs)$ by definition and $\dtw(-Y,\spincs;\F) = \dtop(-Y, \spincs) = -\dbot(Y, \spincs)$ by \cite[Proposition 3.7]{LevineRubermanStrleNonorientable}. Moreover, by \cite[Lemma 3.5]{LevineRubermanStrleNonorientable}, we have $\dtop(Y,\spincs) \le \dbot(Y,\spincs) + b_1(Y)$
Thus,
\[
\dtw(Y,\spincs; \F) + \dtw(-Y, \spincs; \F) \le b_1(Y),
\]
which implies \eqref{eq: reversal-untwisted}.
\end{proof}

Motivated by Proposition \ref{prop: reversal}, we make the following definition:

\begin{definition} \label{def: d-symmetric}
Given a closed, oriented $3$-manifold $Y$ and a torsion spin$^c$ structure $\spincs$ on $Y$, we say that $(Y,\spincs)$ is \emph{$d$-symmetric} if for every summand $A \subset H^1(Y)$ on which the triple cup product vanishes, we have
\begin{equation} \label{eq: d-simple-reversal}
d(-Y,\spincs;M_A) =  -d(Y,\spincs; M_A).
\end{equation}
We say $Y$ is $d$-symmetric if $(Y,\spincs)$ is $d$-symmetric for every torsion spin$^c$ structure $\spincs$ on $Y$.
\end{definition}

Combining \ref{eq: dsh-dbot} and \ref{eq: dsh-dtop} with \cite[Proposition 3.7]{LevineRubermanStrleNonorientable}, we immediately deduce that if both $(Y,\spincs)$ and $(-Y,\spincs)$ are $d$-simple, then they are both $d$-symmetric. (We do not know of an example where $(Y,\spincs)$ is $d$-simple while $(-Y,\spincs)$ is not.)

\subsection{Connected sums}

The behavior of twisted $d$-invariants under connected sums is also potentially more complicated than in the untwisted setting. Given summands $A_1 \subset H^1(Y_1)$ and $A_2 \subset H^1(Y_2)$, $A_1 \oplus A_2$ is naturally a summand of $H^1(Y_1 \conn Y_2)$.  Evidently, if the triple cup product vanishes on each $A_i$, then it vanishes on $A_1 \oplus A_2$ as well.  Adapting the usual proof of additivity of $d$ invariants (see \cite[Theorem 4.3]{OSzAbsolute}, \cite[Proposition 3.8]{LevineRubermanStrleNonorientable}, \cite[Proposition 4.3]{LevineRubermanCorrection}), it is straightforward to see that
\begin{equation} \label{eq: conn-sum-ineq}
\dtw(Y_1 \conn Y_2, \spincs_1 \conn \spincs_2, M_{A_1 \oplus A_2}) \ge \dtw(Y_1, \spincs_1, M_{A_1}) + \dtw(Y_2, \spincs_2, M_{A_2}).
\end{equation}
Proving the reverse inequality requires orientation reversal, which is not available. However, if $Y_1$ and $Y_2$ are $d$-simple, then we have:
\begin{align*}
\dtw(Y_1 \conn Y_2, \spincs_1 \conn \spincs_2, M_{A_1 \oplus A_2}) &\le d(Y_1 \conn Y_2, \spincs_1 \conn \spincs_2, (A_1 \oplus A_2)^{\perp}) \\
&= d(Y_1 \conn Y_2, \spincs_1 \conn \spincs_2, A_1^\perp \oplus A_2^{\perp}) \\
&= d(Y_1, \spincs_1, A_1^{\perp}) + d(Y_2, \spincs_2, A_2^{\perp}) \\
&= \dtw(Y_1, \spincs_1, M_{A_1}) + \dtw(Y_2, \spincs_2, M_{A_2}),
\end{align*}
so equality holds.

\begin{proposition} \label{prop: QHS}
If $Y$ is of the form $Q \conn n(S^1 \times S^2)$, where $Q$ is a rational homology sphere and $n \ge 0$, then $Y$ is $d$-simple and therefore $d$-symmetric. Indeed, if $\spincs = \spinct \conn \spinct_0 \conn \dots \conn \spinct_0$, where $\spinct$ is a spin$^c$ structure on $Q$ and $\spinct_0$ is the unique torsion spin$^c$ structure on $S^1 \times S_2$, then for any subspace $A \subset H^1(Y)$, we have
\[
\dsh(Y, \spincs; M_A) = d(Q, \spinct).
\]
\end{proposition}

\begin{proof}
Clearly, any rational homology sphere is $d$-simple, as is $S^1 \times S^2$. Given a subspace $A \subset \#n(S^1 \times S^2)$, there is a self-diffeomorphism of $\#n(S^1 \times S^2)$ such that the pullback of $A$ can be viewed as $A_1 \oplus \dots \oplus A_n$, where each $A_i$ is a subspace of $H^1$ of the $i\Th$ $S^1 \times S^2$ summand. (That is, any change of basis on $H^1(\#n (S^1 \times S^2))$ can be realized geometrically by handleslides.) The result then follows.
\end{proof}

\subsection{Congruence condition}

Given a closed, oriented $3$-manifold $Y$ and a torsion spin$^c$ structure $\spincs$, there is an invariant $\rho(Y,\spincs) \in \Q/2\Z$, defined to be the congruence class of
\begin{equation}\label{eq:rho}
\frac{c_1(\spinct)^2 - \sigma(W)}{4}
\end{equation}
where $(W,\spinct)$ is any spin$^c$ $4$--manifold with boundary $(Y,\spincs)$. (Here $\sigma(W)$ denotes the signature of $W$.) In this section we prove the following:

\begin{proposition} \label{prop: d-rho}
For any closed, oriented $3$-manifold $Y$, any torsion spin$^c$ structure $\spincs$, and any subspace $A \subset H^1(Y)$ on which the triple cup product vanishes, we have
\[
\dsh(Y,\spincs;M_A) \equiv \rho(Y,\spincs) \pmod {2\Z}.
\]
\end{proposition}

The case where $Y$ is a rational homology sphere was proven by Ozsv\'ath and Szab\'o \cite[Theorem 1.2]{OSzAbsolute}.

\begin{proof}
Suppose $b_1(Y)=n$ and $\rank(A) = k$. In the proof of Theorem \ref{thm: HFi-standard}, we inductively produced a spin$^c$ cobordism $(W_1, \spinct_1) \co (Y_1,\spincs_1) \to (Y,\s)$ by successively attaching $n-k$ $2$-handles along $0$-framed knots. The untwisted homology $\HFi(Y_1,\spincs_1;\F)$ is standard, and the cobordism induces an isomorphism
\[
F^\infty_{W_1,\spinct_1} \co \HFi(Y_1,\spincs_1; \F) \to \HFi(Y, \spincs; M_A).
\]
Since $c_1(\spinct_1)$ is torsion by construction, the grading shift of $F^\infty_{W_1,\spinct_1}$ is equal to $\frac{k-n}{2}$. It follows that
\[
\dtw(Y,\spincs;M_A) \equiv \dtop(Y_1, \spincs_1) - \frac{n-k}{2} \pmod {2\Z}.
\]
By \cite[Lemma 3.5]{LevineRubermanStrleNonorientable},
\[
\dtop(Y_1,\spincs_1) \equiv \dbot(Y_1, \spincs_1) + k \pmod {2\Z}.
\]

Next, we find a cobordism $(W_2,\spinct_2) \co (Y_2,\spincs_2) \to (Y_1, \spincs_1)$, where $Y_2$ is a rational homology sphere, again obtained by successively attaching $k$ $2$-handles along $0$-framed knots. By \cite[Proposition 9.3]{OSzAbsolute}, the map
\[
F^\infty_{W_2,\spinct_2} \co \HFi(Y_2,\spincs_2) \to \HFi(Y_1, \spincs_1).
\]
is injective and takes $\HFi(Y_2,\spincs_2) \cong \Z[U,U^{-1}]$ to the bottom tower in $\HFi(Y_1,\spincs_1)$. Hence,
\[
\dbot(Y_1,\spincs_1) \equiv d(Y_2,\spincs_2) - \frac{k}{2} \pmod {2\Z}.
\]
Combining these congruences, we see that
\begin{align*}
\dtw(Y,\spincs;M_A) &\equiv d(Y_2,\spincs_2) +k - \frac{n}{2}  \pmod {2\Z} \\
\dsh(Y,\spincs;M_A) &\equiv d(Y_2,\spincs_2) \equiv \rho(Y_2,\spincs_2) \pmod{2\Z}.
\end{align*}
Finally, the spin$^c$ cobordisms $(W_1, \spinct_1)$ and $(W_2,\spinct_2)$, each of which has signature $0$, give us $\rho(Y,\spincs) = \rho(Y_1,\spincs_1) = \rho(Y_2,\spincs_2)$, which concludes the proof.
\end{proof}

\section{Cross-sections of open 4-manifolds} \label{sec: cross-section}
\subsection{Topological preliminaries}

Throughout this section, we will be working with open $4$-manifolds obtained as the infinite cyclic covers of homology $S^1 \times S^3$s. For greater generality, we begin by stating the salient algebraic-topology properties of such manifolds, and then work throughout with open manifolds satisfying those properties. (The terminology is motivated by Hughes and Ranicki \cite{hughes-ranicki:ends}, who have a stronger, homotopy-theoretic notion that they call a {\em ribbon}.)

\begin{definition} \label{def: ribbon}
A \emph{homology ribbon} is a smooth, connected, orientable, open $4$-manifold $\tilde X$ with two ends that satisfies the following properties:
\begin{enumerate}
\item \label{item: H3(Xtilde)}
$H_3(\tilde X) \cong \Z$.

\item \label{item: intform}
The intersection form on $H_2(\tilde X) \cong H^2_c(\tilde X)$ vanishes.

\item \label{item: H(Xtilde,end)}
For each end $\epsilon$ of $\tilde X$ and any field $\k$, we have $H_1(\tilde X, \epsilon; \k) \cong H_2(\tilde X, \epsilon; \k) \cong 0$.
\end{enumerate}
We call $\tilde X$ a homology $S^3 \times \R$ if (in addition to the above properties) $H_1(\tilde X) = H_2(\tilde X) = 0$, and a rational homology $S^3 \times \R$ if $H_1(\tilde X;\Q) = H_2(\tilde X;\Q) = 0$.
\end{definition}

\begin{proposition} \label{prop: cover}
Let $X$ be a smooth, closed, oriented $4$-manifold such that $H_*(X) \cong H_*(S^1 \times S^3)$, and let $p \co \tilde X \to X$ denote the universal abelian cover of $X$, with deck transformation group $\Z$. Then $\tilde X$ is a homology ribbon, and $p_*\co H_3(\tilde X) \to H_3(X) \cong \Z$ is an isomorphism.
\end{proposition}

\begin{proof}
For property \ref{item: H3(Xtilde)}, Milnor \cite[Remark 1]{MilnorCyclic} shows that $H_3(\tilde X) \cong H^0(\tilde X) \cong \Z$.

Let $\tau \co \tilde X \to \tilde X$ denote a generator of the deck transformation group. Note that $H_*(\tilde X)$ is a $\Z[t,t^{-1}]$--module, where $t$ acts by $\tau_*$. The Milnor exact sequence
\begin{equation} \label{eq: smith-seq}
\dots \to H_i(\tilde X) \xrightarrow{1-t} H_i(\tilde X) \xrightarrow{p_*} H_i(X) \to H_{i-1}(\tilde X) \to \dots
\end{equation}
implies that $1-t$ is an isomorphism on $H_1(\tilde X)$ and $H_2(\tilde X)$ and zero on $H_3(\tilde X)$. It follows that $p_*\co H_3(\tilde X) \to H_3(X)$ is an isomorphism.

For each integer $m \ge 1$, let $\tilde X \xrightarrow{q_m} X_m \xrightarrow{p_m} X$ denote the intermediate $m$-fold cover of $X$ with deck group $\Z/m$. A standard argument shows that when $m$ is a prime power $p^k$, $H_*(X_m; \Z_p) \cong H_*(S^1 \times S^3;\Z_p)$, and therefore $H_2(X_m;\Q) = 0$ since $H_*(X_m;\Z)$ is finitely generated. In particular, the intersection form on $H_2(X_m;\Z)$ is trivial. Now, given any classes $a,b \in H_2(\tilde X)$, let $\Sigma_a, \Sigma_b$ be closed, oriented, embedded surface representatives that intersect transversally. For $m$ sufficiently large, the restriction of $q_m$ to $\Sigma_a \cup \Sigma_b$ is a diffeomorphism onto its image, so $a \cdot b = q_{m*}(a) \cdot q_{m*}(b) = 0$. This proves property \ref{item: intform}.

For property 3, it is easiest to work with simplicial homology. Choose a finite triangulation of $X$, and lift it to a locally finite triangulation of $\tilde X$. After possibly replacing $\tau$ by $\tau^{-1}$, we may assume that $\tau$ shifts in the direction of the end $\epsilon$. Then
\[
C_*(\tilde X,\epsilon;\k) \cong C_*(\tilde X;\k) \otimes_{\k[t,t^{-1}]} \k[[t,t^{-1}].
\]
Since $\k[[t,t^{-1}]$ is flat as a $\k[t,t^{-1}]$--module, we have
\[
H_*(\tilde X,\epsilon;\k) \cong H_*(\tilde X;\k) \otimes_{\k[t,t^{-1}]} \k[[t,t^{-1}].
\]
Note that $H_*(\tilde X; \k)$ is finitely generated as a $\k[t,t^{-1}]$--module. Since $1-t$ acts as an isomorphism on $H_j(\tilde X; \k)$ for $j=1,2$, we have
\[
H_j(\tilde X; \k) \cong \bigoplus_{l=1}^n \k[t,t^{-1}]/(p_l),
\]
where each $p_l$ is a nonzero, monic polynomial. Since $p_l$ is invertible in $\k[[t,t^{-1}]$, we deduce that
\[
H_j(\tilde X; \k) \otimes_{\k[t,t^{-1}]} \k[[t,t^{-1}] = 0,
\]
as required.
\end{proof}

For the rest of this section, unless otherwise specified, $\tilde X$ will denote an arbitrary homology ribbon, without the requirement that it is the cover of a homology $S^1 \times S^3$.

A \emph{cross-section} of $\tilde X$ is a connected, smoothly embedded, oriented $3$--manifold $Y$ representing a generator of $H_3(\tilde X)$. To find such a cross-section, one can proceed as in Example 3 of the introduction to~\cite{hughes-ranicki:ends}, which treats the case of a manifold with a single end.  Using a proper exhaustion of $\tilde X$, one finds a smooth proper map $f: \tilde X \to \R$ with the ends going to $\pm \infty$.  (The choice of a generator of $H_3(\tilde X) \cong H^1c(\tilde X)$ determines which end goes to $+\infty$.) Then there is a component of the preimage of a regular value that is a cross-section. Denote the closures of the components of $\tilde X \minus Y$ by $L_Y$ and $R_Y$ so that $Y = \partial L_Y = - \partial R_Y$ as an oriented manifold. Note that reversing the orientation of $Y$ (and hence the class in $H_3(\tilde X)$ that $Y$ represents) interchanges the roles of $L_Y$ and $R_Y$: $L_{-Y} = R_Y$ and $R_{-Y} = L_Y$.

If disjoint cross-sections $Y_1$ and $Y_2$ represent the same homology class, we say that $Y_2$ is to the right of $Y_1$, denoted $Y_1 \prec Y_2$, if $Y_2 \subset R_{Y_1}$. This notion depends on which homology class $Y_1$ and $Y_2$ represent; if $Y_1 \prec Y_2$, then $-Y_2 \prec -Y_1$. In what follows, whenever we write $Y_1 \prec Y_2$, we implicitly assume that $Y_1$ and $Y_2$ represent the same homology class.
If $Y_1 \prec Y_2$, let $W(Y_1,Y_2)$ be the closure of $\tilde X \minus (L_{Y_1} \cup R_{Y_2})$; this is an oriented cobordism from $Y_1$ to $Y_2$.

Fix a torsion spin$^c$ structure $\spincs$ on $\tilde X$. By abuse of notation, the restriction of $\spincs$ to any cross-section $Y$ or any cobordism $W(Y_1,Y_2)$ will also be denoted by $\spincs$. If $\tilde X$ is in fact the $\Z$--cover of a homology $S^1 \times S^3$ $X$, then let $\spincs_X$ denote the pullback of the unique spin$^c$ structure on $X$.

We begin with a few basic facts concerning the algebraic topology of cross-sections. First, note that the Mayer--Vietoris sequence for $\tilde X = L_Y \cup R_Y$ shows that $H_3(L_Y) \cong H_3(R_Y) \cong \Z$. Next, consider the long exact sequence on cohomology (both ordinary and compactly supported) for the pair $(L_Y,Y)$:
\begin{equation} \label{eq: H(LY,Y)}
\xymatrix@C=0.25in{
H^0_c(L_Y) \ar[r] \ar[d] & H^0(Y) \ar[r] \ar[d]^{=} & H_c^1(L_Y,Y) \ar[r] \ar[d] & H_c^1(L_Y) \ar[r]^-{j_Y^c} \ar[d]^{\kappa_L} & H^1(Y) \ar[r]^-{\delta_Y^c} \ar[d]^{=} & H^2_c(L_Y,Y) \ar[d] \\
H^0(L_Y) \ar[r]^-{\cong} & H^0(Y) \ar[r]^-{0} & H^1(L_Y,Y) \ar[r] & H^1(L_Y) \ar[r]^-{j_Y} & H^1(Y) \ar[r]^-{\delta_Y} & H^2(L_Y,Y)  }
\end{equation}
Note that $H^0_c(L_Y)=0$ since $L_Y$ is non-compact. By Poincar\'e duality, $H^1_c(L_Y,Y) \cong H_3(L_Y) \cong \Z$. By looking at the same diagram with coefficients in $\Z/p$ for each prime $p$, we deduce that the coboundary $H^0(Y) \to H_c^1(L_Y,Y)$ is an isomorphism, the map $j_Y^c$ is injective, and the quotient $H^1(Y)/H^1_c(L_Y) \cong \im(\delta_Y^c)$ is torsion-free. (That is, $H^1_c(L_Y)$ is a direct summand of $H^1(Y)$.) In particular, $H^1_c(L_Y)$ is a finitely generated, free abelian group; let $b_1^c(L_Y)$ denote its rank. Moreover, the map $\kappa_L \co H^1_c(L_Y) \to H^1(L_Y)$ is injective. (Analogous statements hold with $R_Y$ in place of $L_Y$.)

Next, consider the Mayer--Vietoris sequences (in both ordinary and compactly supported cohomology) for the decomposition $\tilde X = L_Y \cup_Y R_Y$, and the natural maps between them.
\begin{equation} \label{eq: MV}
\xymatrix@C=0.2in{
H^0(Y) \ar[r] \ar[d]^{=}  & \Hc^1(\tilde X) \ar[r]\ar[d]^{\kappa_X} & \Hc^1(L_Y) \oplus \Hc^1(R_Y) \ar[r] \ar@{^{(}->}[d]^{\kappa_L\oplus \kappa_R} & H^1(Y) \ar[d]^{=} \ar[r] & H^2_c(\tilde X) \ar[d]  \\
H^0(Y) \ar[r] & H^1(\tilde X) \ar[r] & H^1(L_Y) \oplus H^1(R_Y) \ar[r] & H^1(Y) \ar[r] & H^2(\tilde X)
}
\end{equation}
(The construction of the upper sequence is most readily carried out if one uses the simplicial version of cohomology with compact supports, as described in~\cite[\S 3.3]{hatcher}; we remark that exactness uses the fact that $Y$ is compact.) Just as before, we deduce that the map $H^1_c(L_Y) \oplus H^1_c(R_Y) \to H^1(Y)$ is injective, meaning that the images of $H^1_c(L_Y)$ and $H^1_c(R_Y)$ in $H^1(Y)$ intersect trivially. However, this map need not be surjective, as in the following example:

\begin{example} \label{ex: YxR}
Suppose $X$ is a homology $S^1 \times S^3$ obtained as the mapping torus of a self-diffeomorphism of some $3$-manifold $Y$. Then $\tilde X \cong Y \times \R$. If we consider $Y = Y \times \{0\}$ as a cross-section of $\tilde X$, it is easy to check that $H^*_c(L_Y) \cong H^*_c(R_Y)  \cong 0$. In particular, the coboundary $H^1(Y) \to H^2(\tilde X)$ in \eqref{eq: MV} is an isomorphism.
\end{example}

In the case where $\tilde X$ is a rational homology $S^3 \times \R$, the situation simplifies considerably, so that we can use ordinary rather than compactly supported cohomology throughout.

\begin{lemma} \label{lemma: HS3xR}
If $\tilde X$ is a homology $S^3 \times \R$, then:
\begin{enumerate}
\item The maps $\kappa_L \co H_c^1(L_Y) \to H^1(L_Y)$ and $\kappa_R \co H_c^1(R_Y) \to H^1(R_Y)$ are isomorphisms, so $j_Y$ and $j_Y^c$ are identified.

\item We have $H^1(Y) \cong H^1(L_Y) \oplus H^1(R_Y) \cong H^1_c(L_Y) \oplus H^1_c(R_Y)$.

\item The sequence
\[
\xymatrix@1{
0 \ar[r] & H^1(L_Y) \ar[r]^-{j_Y} & H^1(Y) \ar[r]^-{\delta_Y} & H^2(L_Y,Y) \ar[r] & 0
}
\]
is short exact and splits, and $H^2(L_Y,Y) \cong H^1(R_Y)$.
\end{enumerate}
\end{lemma}

\begin{proof}
If $\tilde X$ is a rational homology $S^3 \times \R$, then $H^1(\tilde X) = 0$, and $H^2(\tilde X)$ and $H^2_c(\tilde X)$ are both torsion groups. It follows that the maps $H^1(L_Y) \oplus H^1(R_Y) \to H^1(Y)$ and $H^1_c(L_Y) \oplus H^1_c(R_Y) \to H^1(Y)$ in \eqref{eq: MV} are both isomorphisms, so $\kappa_L\oplus \kappa_R$ is as well. Moreover, by the exact sequence for $(\tilde X,R_Y)$ and excision, $H^1(R_Y) \cong H^2(\tilde X,R_Y) \cong H^2(L_Y,Y)$. The restriction map $H^1(R_Y) \to H^1(Y)$ provides a splitting for the short exact sequence.
\end{proof}

Returning to the general case, it is useful to consider one more version of the Mayer--Vietoris sequence, which again is most easily proved using simplicial cohomology as in \cite[\S 25]{munkres:at}. If $\epsilon$ denotes the left end of $X$ corresponding to $L_Y$, we have an exact sequence
\begin{equation} \label{eq: MV-end}
H^1(\tilde X,\epsilon) \to H^1_c(L_Y) \oplus H^1(R_Y) \to H^1(Y) \to H^2(\tilde X,\epsilon).
\end{equation}
In particular, if we take coefficients in $\Q$ and apply property \ref{item: H(Xtilde,end)} from Definition \ref{def: ribbon} together with universal coefficients, we see that
\begin{equation} \label{eq: H1(Y)-end}
H^1(Y;\Q) \cong H^1_c(L_Y;\Q) \oplus H^1(R_Y;\Q).
\end{equation}

Finally, we recall the locally finite homology groups of a (non-compact) polyhedral space $Z$, $\Hlf_*(Z)$. These can be defined in greatest generality using an inverse limit; see Laitinen \cite[Section 2]{laitinen:ends}. When $Z$ is has a locally finite triangulation, it is easiest to use the simplicial version: the chain group $\Clf_i(Z)$ consists of possibly infinite sums of $i$-simplices, and the differential is defined in the usual way. The universal coefficient theorem \cite[Proposition 2.8]{laitinen:ends} relating locally finite homology and compactly supported cohomology takes a slightly unusual form: for any principal ideal domain $R$, there is an exact sequence
\begin{equation} \label{eq: Hlf-univ-coeff}
0 \to \Ext(H_c^{n+1}(Z),R) \to \Hlf_n(Z;R) \to \Hom(H_c^n(Z),R) \to 0.
\end{equation}
Additionally, if $Z$ is an $n$-dimensional manifold, then there is a Poincar\'e duality isomorphism $\Hlf_k(Z) \cong H^{n-k}(Z)$ \cite[Theorem 3.1]{laitinen:ends}.

\subsection{Correction terms}

We will be considering Heegaard Floer homology with coefficients in the module
\begin{equation} \label{eq: LLY}
\LL_Y := M_{H^1_c(L_Y)} = \F[H^1(Y)/H^1_c(L_Y)] = \F[\im(\delta_Y^c)].
\end{equation}
The key observation is the following:

\begin{proposition} \label{prop: HFi-section}
Let $\tilde X$ be a homology ribbon and let $\spincs$ be a torsion spin$^c$ structure on $\tilde X$. For any cross-section $Y$, the restriction of the triple cup product form on $H^1(Y)$ to the image of $H^1_c(L_Y)$ vanishes identically. Therefore,
\[
\HF^\infty(Y,\spincs; \LL_Y) \cong \Lambda^*(H^1_c(L_Y)) \otimes \F[U,U^{-1}]
\]
as a $\Lambda^*(H_1(Y)/(H^1_c(L_Y)^\perp)) \otimes \F[U,U^{-1}]$--module. Moreover, we may naturally identify $H_1(Y)/(H^1_c(L_Y))^\perp$ with $\Hlf_1(L_Y)/\Tors$. Analogous statements hold with $R_Y$ in place of $L_Y$.
\end{proposition}

\begin{proof}
First, note that there is a fundamental class $[L_Y,Y] \in \Hlf_4(L_Y,Y)$ which maps to the fundamental class $[Y] \in H_3(Y)$ under the boundary map. (If we are given a locally finite triangulation of $L_Y$ with $Y$ as a subcomplex, the fundamental class is given as the sum of all the $4$-simplices, with signs determined by the orientation, and the boundary map $\Hlf_4(L_Y,Y) H_3(Y)$ is just the usual simplicial boundary.) We then argue just as we would if $L_Y$ were a compact manifold: For any  $\alpha_1, \alpha_2, \alpha_3 \in H^1_c(L_Y)$, we have:
\begin{align*}
\gen{i^*(\alpha_1) \cup i^*(\alpha_2) \cup i^*(\alpha_3), [Y]}
&= \gen{i^*(\alpha_1 \cup \alpha_2 \cup \alpha_3), [Y]} \\
&= \gen{\alpha_1 \cup \alpha_2 \cup \alpha_3, i_*([Y])} \\
&= \gen{\alpha_1 \cup \alpha_2 \cup \alpha_3, i_*(\partial([L_Y,Y]))} \\
&= 0.
\end{align*}
The second and third lines make use of the pairing between $H^3_c(L_Y)$ and $\Hlf_3(L_Y)$ given in \eqref{eq: Hlf-univ-coeff} (taking $R=\Z$). The same exact sequence also yields the identification
\[
H_1(Y)/(H^1_c(L_Y))^\perp \cong \Hlf_1(L_Y)/\Tors.
\]

Finally, the statement about $\HF^\infty$ follows directly from Theorem \ref{thm: HFi-standard}.
\end{proof}

Proposition \ref{prop: HFi-section} allows us to consider the twisted correction term $\dtw(Y,\spincs; \LL_Y)$ and the shifted version $\dsh(Y,\spincs; \LL_Y)$. By definition, we have
\begin{equation} \label{eq: dsh(Y,LLY)}
\dsh(Y,\spincs; \LL_Y) = \dtw(Y,\spincs; \LL_Y) + \frac{ b_1(Y) - 2b_1^c(L_Y)}{2}.
\end{equation}
Observe that when $\tilde X$ is a homology $S^3 \times \R$, Lemma \ref{lemma: HS3xR} implies that $b_1^c(L_Y) = b_1(L_Y)$ and $\chi(L_Y) = b_1(Y) - 2b_1(L_Y)$, and hence
\begin{equation} \label{eq: dsh(Y,LLY)-HS3xR}
\dsh(Y,\spincs; \LL_Y) = \dtw(Y,\spincs; \LL_Y) + \frac{ \chi(L_Y)}{2}.
\end{equation}

Our choice of coefficients in $\LL_Y$ (as opposed to the analogous construction using the cohomology of $R_Y$) is justified by the following lemma:

\begin{lemma} \label{lemma: compatible}
Let $\tilde X$ be a homology ribbon. Suppose $Y_1, Y_2$ are disjoint, homologous cross-sections of $\tilde X$ with $Y_1 \prec Y_2$, and let $W = W(Y_1,Y_2)$ be the cobordism between them. Then $\LL_{Y_1}(W) \cong \LL_{Y_2}$.
\end{lemma}

\begin{proof}
According to \eqref{eq: LLY}, we are trying to show that
\[
\F[\im(\delta_{Y_1}^c)] \otimes_{\F[H^1(Y_1)]} \F[K(W)] \cong \F[\im(\delta_{Y_2}^c)].
\]
We prove this by constructing an exact sequence
\[
H^1(Y_1) \to \im(\delta_{Y_1}^c) \oplus K(W) \to \im(\delta_{Y_2}^c) \to 0.
\]
as follows.

First, by property \ref{item: intform} in Definition \ref{def: ribbon}, the intersection form on $H_2(W)$ vanishes identically. Therefore, the map $j_W\co H^2(W,\partial W) \to H^2(W)$ vanishes identically, meaning that $K(W)  = H^2(W, \partial W)$.

Consider the following commutative diagram, whose rows and columns are exact:
\[
\xymatrix{
H^1(Y_1 \cup Y_2,Y_2) \ar@{^{(}->}[r] \ar[d]^{=} & H^1(Y_1 \cup Y_2) \ar@{->>}[r] \ar[d]^{\gamma} & H^1(Y_2) \ar[r]^-{0} \ar[d]^{\delta_{Y_2}^c} & H^2(Y_1 \cup Y_2,Y_2) \ar[d]^{=} \\
H^1(Y_1 \cup Y_2,Y_2) \ar[r]^-{f} & H^2_c(L_{Y_2}, Y_1 \cup Y_2) \ar[r]^-{g} \ar[d] & H^2_c(L_{Y_2},Y_2) \ar[r]^-{h} \ar[d] & H^2(Y_1 \cup Y_2,Y_2) \\
& H^2_c(L_{Y_2}) \ar[r]^-{=} & H^2_c(L_{Y_2}) &
}
\]
We easily deduce that $\im(\delta^c_{Y_2}) \subset \im(g)$ and that $g^{-1}(\im(\delta^c_{Y_2}))  = \im(\gamma)$; thus, the middle row gives rise to an exact sequence
\[
\xymatrix{
H^1(Y_1 \cup Y_2, Y_2) \ar[r]^-{f} & \im(\gamma) \ar[r]^-{g} & \im(\delta^c_{Y_2}) \ar[r] & 0.
}
\]
Of course, $H^1(Y_1 \cup Y_2, Y_2) \cong H^1(Y_1)$.

Next, the Mayer--Vietoris sequence for the decomposition $(L_{Y_2}, Y_1 \cup Y_2) = (L_{Y_1},Y_1) \cup (W, \partial W)$ shows that
\[
H^2_c(L_{Y_2}, Y_1 \cup Y_2) \cong H^2_c(L_{Y_1},Y_1) \oplus H^2(W,\partial W).
\]
Under this identification, it is easy to see that the image of $\gamma$ is identified with $\im(\delta^c_{Y_1}) \oplus H^2(W,\partial W)$, as required.
\end{proof}

\begin{proposition} \label{prop: d-relation}
Let $\tilde X$ be a homology ribbon and let $\spincs$ be a torsion spin$^c$ structure on $X$. Suppose $Y_1, Y_2$ are disjoint, homologous cross-sections of $\tilde X$ with $Y_1 \prec Y_2$. Then
\begin{equation} \label{eq: dsh-relation}
\dsh(Y_1, \spincs; \LL_{Y_1}) \le  \dsh(Y_2, \spincs; \LL_{Y_2}).
\end{equation}
\end{proposition}

\begin{proof}
It suffices to assume that the cobordism  $W = W(Y_1,Y_2)$ is given by a single handle attachment. To be precise, let $L_{Y_1}'$ denote the union of $L_{Y_1}$ with the closure of a product neighborhood of $Y$, and assume that
\[
L_{Y_2} = L_{Y_1}' \cup \text{$k$-handle},
\]
where $k \in \{1,2,3\}$. When $k=2$, because the intersection form on $\tilde X$ vanishes, $L_{Y_2}$ cannot be obtained by attaching a $2$-handle to a rationally null-homologous curve with nonzero framing.

The proof proceeds as follows. By Lemma \ref{lemma: compatible}, the cobordism $W$ induces maps
\[
F^\circ_{W,\spincs} \co \HF^\circ(Y_1, \spincs; \LL_{Y_1}) \to \HF^\circ(Y_2, \spincs; \LL_{Y_2}).
\]
Since $c_1(\spincs)$ is torsion, the grading shift of $F^\circ_{W,\spincs}$ is equal to $-\frac{\chi(W)}{2}$. Note that $\chi(W) = 1$ when $k=2$ and $-1$ otherwise.

We will show in each case that $F^\infty_{W,\spincs}$ descends to an isomorphism
\[
\QQ \HFi(Y_1, \spincs; \LL_{Y_1}) \to \QQ \HFi(Y_2, \spincs; \LL_{Y_2}).
\]
Thus, by the usual argument,
\begin{equation} \label{eq: d-relation}
d(Y_1,\spincs; \LL_{Y_1}) \le d(Y_2,\spincs; \LL_{Y_2}) + \frac{\chi(W)}{2}.
\end{equation}
At the same time, we will see in each case that
\begin{equation} \label{eq: chi(W)}
\chi(W) =  \left(b_1(Y_2) - 2 b_1^c(L_{Y_2}) \right) - \left(b_1(Y_1) - 2 b_1^c(L_{Y_1}) \right),
\end{equation}
from which \eqref{eq: dsh-relation} follows.

We now consider the different values for $k$. A summary can be found in Table \ref{table: handles}.

\begin{table}
\begin{tabular}{|c|c|c|c|c|c|}
  \hline
  Handle type & $\Delta b_1(Y)$ & $\Delta b_1^c(L_Y)$ & $\Delta b_1^c(R_Y)$ & $\chi(W)$ & $F^\infty_{W,\spincs}$ \\ \hline
  $1$ & $+1$ & $+1$ & $0$ & $-1$ & injective \\
  $2$, $[K]$ torsion & $+1$ & $0$ & $+1$ & $+1$ & isomorphism  \\
  $2$, $[K]$ non-torsion & $-1$ & $-1$ & $0$ & $+1$ & surjective  \\
  $3$ & $-1$ & $0$ & $-1$ & $-1$ & isomorphism \\
  \hline
\end{tabular}

\caption{Summary of handle additions in the proof of Proposition \ref{prop: d-relation}. The convention is that $\Delta b_1(Y) = b_1(Y_2) - b_1(Y_1)$, etc. It is easy to see that $\Delta b_1(Y) - 2 \Delta b_1^c(L_Y) = \chi(W)$ and that $\Delta b_1(Y) = \Delta b_1^c(L_Y) + \Delta b_1^c(R_Y)$.} \label{table: handles}
\end{table}

\begin{enumerate}
\item
If $L_{Y_2} \cong L_{Y_1}' \cup 1$-handle, then $Y_2 \cong Y_1 \conn S^1 \times S^2$, so $b_1(Y_2) = b_1(Y_1)+1$. By looking at the exact sequence on cohomology for the pair $(L_{Y_2}, L_{Y_1})$, we see that $H^1_c(L_{Y_2}) \cong H^1_c(L_{Y_1}) \oplus \Z$, where a generator of the $\Z$ factor maps to the Poincar\'e dual of $[\{\pt\} \times S^2]$ in $H^1(Y_2)$. Hence, $b_1^c(L_{Y_2})=b_1^c(L_{Y_1})+1$, so \eqref{eq: chi(W)} holds.

As in \cite[Section 4.3]{OSz4Manifold}, we have
\[
\HFi(Y_2,\spincs_2;\LL_{Y_2}) \cong \HFi(Y_1,\spincs_1;\LL_{Y_1})[\tfrac12] \oplus \HFi(Y_1,\spincs_1;\LL_{Y_1})[-\tfrac12],
\]
where the action of the generator of $H_1(S^1 \times S^2)$ takes the first summand to the second, and the map $F^\infty_{W,\spincs}$ is given by the inclusion of the first summand.

\item
If $L_{Y_2}$ is obtained by attaching a $2$-handle to $L'_{Y_1}$ along a curve $K \subset Y_1$, there are two possibilities.

If $K$ represents a torsion class in $H_1(Y)$, then the $2$-handle must be attached along the rational longitude of $K$; otherwise, there would be a closed surface in $\tilde X$ with nontrivial self-intersection, which violates our assumptions. Thus, $Y_2$ is obtained by $0$-surgery on $K$, and $b_1(Y_2) = b_1(Y_1)+1$. Moreover, the inclusion $L_{Y_1} \to L_{Y_2}$ induces an isomorphism $H^1_c(L_{Y_2}) \to H^1_c(L_{Y_1})$. By Proposition \ref{prop: HFi-surgery} (and its extension to the rationally nulhomologous case in Remark \ref{rmk: rationally-nulhom}), $F^\infty_{W, \spincs}$ is an isomorphism that respects the $H_1$ actions.

If $K$ represents a nontorsion class in $H_1(Y_1)$, then $b_1(Y_2) = b_1(Y_1)-1$. Equation \eqref{eq: H1(Y)-end} implies that is a class $\alpha \in H^1_c(L_{Y_1};\Q)$ such that $\gen{\alpha|_{Y_1}, [K]} =1$, and therefore $[K]$ is also nontorsion in $\Hlf_1(L_{Y_1})$. The restriction map $H^1_c(L_{Y_2}) \to H^1_c(L_{Y_1})$ is injective, with image equal to the set of elements that evaluate to $0$ on $[K]$, so $b_1^c(L_{Y_2})= b_1^c(L_{Y_1}) -1$. Just as in the untwisted case (see \cite[Proposition 9.3]{OSzAbsolute}), $F^\infty_{W,\spincs}$ descends to an isomorphism
\[
\HFi(Y_1,\spincs;\LL_{Y_1})/([K] \cdot \HFi(Y_1,\spincs;\LL_{Y_1})) \xrightarrow{\cong} \HFi(Y_2,\spincs;\LL_{Y_2}),
\]
which respects the $H_1$ actions.

\begin{figure}
\labellist \small
 \pinlabel {{\color{red} $\alpha_0$}} [t] at 60 47
 \pinlabel {{\color{blue} $\beta_0$}} [b] at 60 60
 \pinlabel {$z_0$} at 60 100
 \pinlabel {$a$} [bl] at 29 60
 \pinlabel {$b$} [br] at 91 60
\endlabellist
\includegraphics{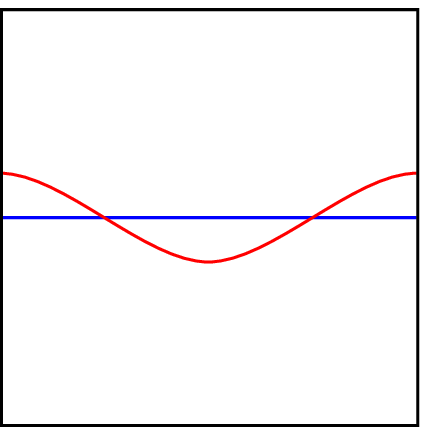}
\caption{Standard Heegaard diagram $(T^2, \alpha_0, \beta_0, z_0)$ for $S^1 \times S^2$.}
\label{fig: S1xS2}
\end{figure}

\item
Suppose $L_{Y_2}$ is obtained by attaching a $3$-handle to $L'_{Y_1}$ along an embedded, non-separating sphere $S \subset Y_2$, which necessarily represents a primitive homology class. Then $Y_1 \cong Y_2 \conn S^1 \times S^2$, so $b_1(Y_2) = b_1(Y_1)-1$, and $H^1_c(L_{Y_2}) \cong H^1_c(L_{Y_1})$. As seen in Section \ref{sec: HF-prelim}, we have $\LL_{Y_2} \cong \LL_{Y_1} / (1-t) \cong \LL_{Y_1}(W)$, where $t$ is the class in $H^1(Y_1)$ Poincar\'e dual to $[S]$.

By \cite[Lemma 4.11]{OSz4Manifold}, we may represent $Y_1$ by a split Heegaard diagram
\[
(\Sigma', \bm\alpha', \bm\beta', z') = (\Sigma, \bm\alpha, \bm\beta, z) \conn (T^2, \alpha_0, \beta_0, z_0),
\]
where $(\Sigma, \bm\alpha, \bm\beta, z)$ represents $Y_2$, $(T^2, \alpha_0, \beta_0, z_0)$ is a standard diagram for $S^1 \times S^2$ as shown in Figure \ref{fig: S1xS2}, and the connected sum is taken in the regions containing the basepoints. The curves $\alpha_0, \beta_0$ meet in two points $a, b$. For any $\x \in \T_\alpha \cap \T_\beta$, there are a pair of holomorphic bigons $\phi^{\pm}_\x \in \pi_2(\x \times \{a\}, \x \times \{b\})$. We may choose the additive assignment such that for each $\x$, the disks $\phi^+_\x$ and $\phi^-_\x$ contribute $1$ and $t$, respectively, in the differential. Just as in Ozsv\'ath and Szab\'o's proof of the K\"unneth formula for connected sums \cite[Theorem 6.2]{OSzProperties}, with respect to a sufficiently stretched complex structure, $\CFi(\Sigma', \bm\alpha', \bm\beta', z'; \LL_{Y_1})$ is then isomorphic to the mapping cone
\[
\CFi(\Sigma, \bm\alpha, \bm\beta, z; \LL_{Y_2})[t^{\pm1}] \xrightarrow{1-t} \CFi(\Sigma, \bm\alpha, \bm\beta, z; \LL_{Y_2})[t^{\pm1}]
\]
where the two copies correspond to $a$ and $b$ respectively. The map $F^\infty_{W, \spincs}$ is given on the chain level by setting $t=1$ and projecting onto the second factor. It follows that
\[
F^\infty_{W, \spincs} \co \HFi(Y_1, \spincs; \LL_{Y_1}) \to \HFi(Y_2, \spincs; \LL_{Y_2})
\]
is an isomorphism.
\end{enumerate}

In each of the three cases, it is easy to see that the induced maps on $\QQ \HF^\infty$ are isomorphisms and that \eqref{eq: chi(W)} holds, as required.
\end{proof}

\begin{proposition} \label{prop: d-even}
If $\tilde X$ is a homology $S^3 \times \R$, then $\dsh(Y, \spincs; \LL_{Y})$ is an {\em even} integer.
\end{proposition}
\begin{proof}
By Proposition~\ref{prop: d-rho}, we know that $\dsh(Y, \spincs; \LL_{Y}) \equiv \rho(Y,\spincs) \pmod {2\Z}$, where $\rho(Y,\spincs)$ is defined by Equation \eqref{eq:rho}. Since the spin$^c$ structure $\spincs$ is a spin structure, we may take the manifold $W$ in \eqref{eq:rho} to be a spin manifold, and so the term $c_1(\spinct)^2 =0$. As in the proof of~\cite[Theorem  3.4]{ruberman-saveliev:4tori}, the signature of $W$ is the same as the signature of the open manifold
\[
W_\infty = W \cup_Y R_Y.
\]
The vanishing of the homology of $\tilde X$ together with property \ref{item: H(Xtilde,end)} from Definition \ref{def: ribbon} implies that the intersection form on the spin manifold $W_\infty$ is unimodular, and hence van der Blij's theorem~\cite[\S 5]{milnor-husemoller} says that its signature is divisible by $8$.
\end{proof}

\begin{remark}
By turning the cobordism $W$ around, it is also easy to see how the quantity $b_1^c(R_Y)$ behaves: we see that $b_1^c(R_{Y_2}) - b_1^c(R_{Y_1})$ equals $0$ in the case of a $1$-handle addition or a $2$-handle addition along a non-torsion curve, $1$ in the case of a $2$-handle addition along a torsion curve, and $-1$ in the case of a $3$-handle addition. In particular, we see from Table \ref{table: handles} that the quantity
\[
b_1(Y) - b_1^c(L_Y) - b_1^c(R_Y)
\]
is independent of the choice of cross-section $Y$. The Mayer--Vietoris sequence (top row of \eqref{eq: MV}) shows that this quantity equals the rank of the coboundary map $H^1(Y) \to H^2_c(\tilde X)$.
\end{remark}

\subsection{Invariants for homology \texorpdfstring{$S^1 \times S^3$s}{S3xS1s}}

We are now finally able to prove the main theorem from the introduction, which we restate as follows:

\begin{theorem} \label{thm: dfour}
Let $X$ be an oriented homology $S^1 \times S^3$, let $\tilde X$ be its infinite cyclic cover, and let $\spincs_X$ be the spin$^c$ structure on $\tilde X$ pulled back from $X$. Then for any cross-section $Y$ of $\tilde X$, the shifted correction term $\dsh(Y, \spincs_X; \LL_Y)$ depends only on the homology class of $Y$ in $H_3(\tilde X)$ or equivalently on its image $\hgen \in H_3(X)$. We denote this number by $\dfour(X, \hgen)$; it is an invariant of $X$ under orientation-preserving diffeomorphisms that preserve the choice of homology class.
\end{theorem}

\begin{proof}
Fix a generator for $H_3(\tilde X)$. Let $\tau$ be a generator of the deck transformation group such that for any two cross-sections $Y, Y'$ representing the fixed generator, $\tau^{-n}(Y) \prec Y' \prec \tau^n(Y)$ for all $n$ sufficiently large. For any $n \in \Z$, note that
\[
\dsh(\tau^n(Y), \spincs_X; \LL_{\tau^n(Y)}) = \dsh(Y, \spincs_X; \LL_Y),
\]
since the spin$^c$ structure $\spincs_X$ on $\tilde X$ is $\tau$-invariant and the deck transformation $\tau^n$ takes $L_Y$ to $L_{\tau^n(Y)}$. Thus, by Proposition \ref{prop: d-relation}, we have
\[
\dsh(Y, \spincs_X; \LL_Y) \le \dsh(Y', \spincs; \LL_{Y'}) \le \dsh(Y, \spincs_X; \LL_Y)
\]
and hence equality holds.
\end{proof}

Next, we prove the symmetries stated in Proposition \ref{prop: dfour-symmetries}. It is more convenient to work in the more general setting of open manifolds. Given a homology ribbon $\tilde X$ equipped with a spin$^c$ structure $\spincs$, and any cross section $Y$ of $\tilde X$, define
\begin{equation} \label{eq: d(X,Y)}
\dfour(\tilde X,Y) = \dsh(Y,\spincs; \LL_Y).
\end{equation}
(For convenience, we suppress the spin$^c$ structure $\spincs$ from the notation.) When $\tilde X$ is the $\Z$ cover of a homology $S^1 \times S^3$ $X$ and $\spincs = \spincs_X$, then by definition $\dfour(\tilde X,Y) = \dfour(X,\hgen)$.

There are two possible orientation changes to consider.

\begin{itemize}
\item If we leave the orientation on $\tilde X$ fixed but change the orientation of $Y$, the roles of $L_Y$ and $R_Y$ are interchanged: $L_{-Y} = R_Y$ and $R_{-Y} = L_Y$. According to our definition, we have
\begin{equation} \label{eq: d(X,-Y)}
\dfour(\tilde X,-Y) = \dsh(-Y, \spincs; \RR_Y),
\end{equation}
where $\RR_Y = \F[H^1(Y)/H^1(R_Y)]$.

\item If we reverse both the orientations of both $\tilde X$ and $Y$, then the roles of $L_Y$ and $R_Y$ do not change, since $\partial(-L_Y) = -Y$. Thus, we may write
\begin{equation} \label{eq: d(-X,-Y)}
\dfour(-\tilde X,-Y) = \dsh(-Y, \spincs; \LL_Y).
\end{equation}
Combining this argument with the previous one, we deduce that
\begin{equation} \label{eq: d(-X,Y)}
\dfour(-\tilde X,Y) = \dsh(Y, \spincs; \RR_Y).
\end{equation}
\end{itemize}

\begin{proof}[Proof of Proposition \ref{prop: dfour-symmetries}]
Suppose that $Y$ is a cross-section of a homology ribbon $\tilde X$. If $Y$ is a rational homology sphere, then $\dfour(X,Y) = d(Y,\spincs)$. By inspecting equations \eqref{eq: d(X,Y)} through \eqref{eq: d(-X,Y)}, it is immediate that
\begin{equation} \label{eq: d-section-QHS}
\dfour(\tilde X,Y) = \dfour(-\tilde X,Y) = -\dfour(\tilde X,-Y) = -\dfour(-\tilde X,-Y).
\end{equation}
Likewise, when $Y$ is merely $d$-symmetric, we obtain
\begin{equation} \label{eq: d-symmetric-section}
\dfour(-\tilde X,-Y) = -\dfour(\tilde X,Y) \quad \text{and} \quad \dfour(-\tilde X,Y) = -\dfour(\tilde X,-Y).
\end{equation}
These translate to \eqref{eq: dfour-QHS} and \eqref{eq: dfour-symmetric}, respectively.

Finally, if $X$ is the mapping torus of a diffeomorphism $\phi \co Y \to Y$ and $\tilde X$ is its universal cover, then $\tilde X \cong Y \times \R$. As seen in Example \ref{ex: YxR}, we have $H^1_c(L_Y) = H^1_c(R_Y) = 0$, and therefore the coefficient modules $\LL_Y$ and $\RR_Y$ are both simply $\HomRing{Y}$. Equations \eqref{eq: d(X,Y)} through \eqref{eq: d(-X,Y)} yield \eqref{eq: dfour-fibered}.
\end{proof}

The following proposition is an immediate consequence of Proposition \ref{prop: d-relation} and equation \eqref{eq: d-symmetric-section}:

\begin{proposition} \label{prop: sandwich}
Let $\tilde X$ be a homology ribbon and let $\spincs$ be a torsion spin$^c$ structure on $\tilde X$. Suppose $Y_1 \prec Y_2$ are disjoint cross sections of $\tilde X$, and $(Y_1, \spincs)$ and $(Y_2, \spincs)$ are both $d$-symmetric. Then $\dfour(\tilde X,Y_1) = \dfour(\tilde X,Y_2)$. Moreover, if $Y'$ is any other cross-section with $Y_1 \prec Y' \prec Y_2$, then
\[
\dfour(-\tilde X,-Y') = -\dfour(\tilde X,Y') = -\dfour(\tilde X,Y_1).
\]
\end{proposition}

This result is useful for obstructing the presence of $d$-symmetric cross-sections (e.g. rational homology spheres) in the ends of exotic $\R^4$s, as in the following example.

\begin{example} \label{ex: exoticR4}
Let $K$ denote the positive, untwisted Whitehead double of the right-handed trefoil, let $Y = S^3_0(K)$, and let $W$ be obtained by attaching a $0$-framed $2$-handle to $D^4$ along $K$, so that $\partial W = Y$.
Let $Z$ denote the complement of a topological slice disk for $K$, with $\pi_1(Z) = \Z$. We may choose a smooth structure on $Z_0 = Z \minus \{\pt\}$; then $Z_0$ is an open, smooth $4$-manifold with $\partial Z_0 = Y$. Then $R = W \cup_Y {-Z_0}$ is an exotic $\R^4$, and $\tilde X = (W \minus B^4) \cup_Y {-Z_0}$ is an exotic $S^3 \times \R$ with one end smoothly modeled on $S^3 \times (-\infty, 0]$. Since $b_1(L_Y)=0$, we have $\LL_Y = \F[H^1(Y)]$. As seen in Example \ref{ex: trefoil}, $\dfour(\tilde X,Y) = 0$ and $\dfour(-\tilde X,-Y) = 2$. If the generator of $H_3(Z_0)$ were represented by any $d$-symmetric manifold, this would contradict Proposition \ref{prop: sandwich}.
\end{example}
\begin{remark}\label{R:exotic}
The existence of an exotic $\R^4$ not containing a homology sphere arbitrarily far out in its end seems to be $4$-manifold folklore; compare~\cite[Page 96, Remark 1]{kirby:4-manifolds}. The proof depends on Donaldson's diagonalization theorem. Bob Gompf pointed out to us that the extension of Donaldson's theorem to non-simply connected manifolds~\cite{donaldson:orientation} can be used to show that there is no rational homology sphere arbitrarily far out in the end.
\end{remark}

\begin{example} \label{ex: T3-partial}
The three-torus $T^3$ embeds in $\R^4$, so it occurs as a cross-section of $S^3 \times \R$. By Lemma \ref{lemma: HS3xR}, we have $H^1(T^3) = H^1(L_{T^3}) \oplus H^1(R_{T^3}) = H^1_c(L_{T^3}) \oplus H^1_c(R_{T^3})$. Because the triple cup product vanishes on each summand, one summand must have rank $1$ and the other rank $2$; by varying the orientations, we may interchange them. Because $S^3 \prec T^3 \prec S^3$, we deduce in either case that $\dsh(T^3,\spincs; \LL_{T^3})=0$. As we saw in Example \ref{ex: T3}, this means that for any subspace $A \subset H^1(T^3)$ of rank $1$ or $2$, we have $\dsh(T^3,\spincs;M_A)=0$.

On the other hand, let $X$ be a homology $S^1 \times S^3$ obtained as the mapping torus of a self-diffeomorphism of $Y=T^3$. (Such manifolds play a key role in the construction of the Cappell--Shaneson homotopy spheres \cite{CappellShaneson4Manifolds,cappell-shaneson:knots}.) Then $H^1_c(L_{T^3}) = H^1_c(R_{T^3}) = 0$. From Example \ref{ex: T3}, we deduce that $\dfour(X,y) = \dfour(-X,-y) = 2$. Hence, $X$ does not admit any $d$-symmetric cross-section. (Of course, because any cross-section of $X$ admits a degree $1$ map to $T^3$ and therefore has nonvanishing triple cup product, cross-sections of the form $Q \conn n(S^1 \times S^2)$ are are automatically excluded.)
\end{example}

\section{Applications to knotted spheres} \label{sec: 2knots}

In this section, we use the invariants defined above to study Seifert surfaces for $2$-knots in $S^4$. Given a smoothly embedded, oriented $2$-sphere $\Sigma$ in $S^4$, a Seifert surface is a smoothly embedded, compact, connected, oriented $3$-manifold with boundary $\Sigma$. Let $X(\Sigma)$ denote the surgered manifold $S^4 \minus \nbd(\Sigma) \cup (D^3 \times S^1)$, which is a homology $S^1 \times S^3$. Any Seifert surface of $\Sigma$ can be capped off to be a cross-section of $X(\Sigma)$. (In a slight abuse of notation, if $Y \minus B^4$ occurs as a Seifert surface of $\Sigma$, we will sometimes say that $\Sigma$ has $Y$ as a Seifert surface.) The homology class $\hgen$ of a capped-off Seifert surface $Y$ in $H_3(X)$ is determined by the orientation of $\Sigma$; therefore, we define $\dknot(\Sigma) = \dfour(X(\Sigma), \hgen)$, which is an invariant of the smooth isotopy class of $\Sigma$.

Let $\Sigma^r$ denote $\Sigma$ with reversed orientation, and let $\bar \Sigma$ denote the image of $\Sigma$ under a reflection of $S^4$. If $Y \subset X(\Sigma)$ is a capped-off Seifert surface for $\Sigma$, then
\begin{align*}
\dknot(\Sigma^r) &= \dfour(X(\Sigma), -\hgen) \\
\dknot(\bar\Sigma) &= \dfour(-X(\Sigma), \hgen) \\
\dknot(\bar\Sigma^r) &= \dfour(-X(\Sigma), -\hgen).
\end{align*}
The four numbers $\dknot(\Sigma)$, $\dknot(\Sigma^r)$, $\dknot(\bar\Sigma)$, and $\dknot(\bar\Sigma^r)$ may \emph{a priori} all be different. We say that say that $\Sigma$ is \emph{invertible}, \emph{positive amphicheiral} or \emph{negative amphicheiral} if $\Sigma$ is smoothly isotopic to $\Sigma^r$, $\bar\Sigma$, or $\bar\Sigma^r$, respectively; the $\dknot$ invariant can thus be used to obstruct such symmetries. Moreover, the symmetries from Proposition \ref{prop: dfour-symmetries} each translate to a symmetry of the $2$-knot invariants. For instance, if $\Sigma$ has a Seifert surface $Y$ that is $d$-symmetric, then
\begin{equation} \label{eq: dknot-symmetric}
\dknot(\Sigma) = -\dknot(\bar\Sigma^r) \quad \text{and} \quad \dknot(\Sigma^r) = -\dknot(\bar\Sigma).
\end{equation}
In particular, if $\Sigma$ is a \emph{ribbon knot} (i.e., bounds an immersed $3$-ball with ribbon singularities), a theorem of Yanagawa \cite{YanagawaRibbon1} states that $\Sigma$ has a Seifert surface diffeomorphic to $\#n(S^1 \times S^2) \minus B^3$ for some $n$; it follows that
\[
\dknot(\Sigma) = \dknot(\Sigma^r) = \dknot(\bar \Sigma) = \dknot(\bar \Sigma^r) = 0.
\]
Likewise, if $\Sigma$ is a fibered $2$-knot with capped-off fiber $Y$, then
\[
\dknot(\Sigma) = \dknot(\bar\Sigma) = \dsh(Y,\spincs_X; \HomRing{Y}) \quad \text{and} \quad
\dknot(\Sigma^r) = \dknot(\bar\Sigma^r) = \dsh(-Y,\spincs_X; \HomRing{Y}).
\]

\begin{example} \label{ex: irreversible}
If $\Sigma$ is the $5$-twist-spin of the right-handed trefoil, then $\Sigma$ has the Poincar\'e homology sphere as a fiber \cite[p. 306]{Rolfsen}. Hence, $\dknot(\Sigma)= d(\bar\Sigma) = 2$ and $\dknot(\Sigma^r) = d(\bar\Sigma^r) = -2$. We deduce that $\Sigma$ is neither reversible (which was also proven by Gordon \cite{GordonReversibility}) nor negative amphicheiral.
\end{example}

\begin{example} \label{ex: no-symmetric-seifert}
Let $\Sigma$ be the $6$-twist-spin of the right-handed trefoil $K$ and $Y$ the fiber, which is the $6$-fold cyclic branched cover of $K$. As explained in \cite[p. 307]{Rolfsen}, $Y$ can be obtained by $(0,0)$ surgery on the positive Whitehead link, or equivalently by $(0,0,-1)$ surgery on the Borromean rings. (Hence, $Y$ has an alternate description as the circle bundle of Euler number $-1$ over the torus.) Likewise, $-Y$ can be obtained by $(0,0,1)$ surgery on the Borromean rings. Let $\spincs$ denote the unique torsion spin$^c$ structure on $Y$. Ozsv\'ath and Szab\'o \cite[Lemma 8.7]{OSzAbsolute} proved that $\dtw(-Y,\spincs; \HomRing{-Y}) = -1$, and therefore $\dknot(\Sigma^r) = -2$. A very similar computation (following the proofs of \cite[Lemmas 8.6 and 8.7]{OSzAbsolute}) shows that $\dtw(Y,\spincs; \HomRing{Y}) = 1$, so $\dknot(\Sigma)=0$. By \eqref{eq: dfour-symmetric}, it follows that $\Sigma$ does not have any Seifert surface that is $d$-symmetric, such as any manifold of the form $Q \conn n(S^1 \times S^2)$.
\end{example}

\begin{remark}
Just as with classical knots, the degree of the Alexander polynomial $\Delta(\Sigma)$ provides a lower bound on $b_1$ of any Seifert surface for $\Sigma$. In particular, if $\Sigma$ admits a Seifert surface that is a rational homology sphere, then $\Delta(\Sigma)=1$. It would thus be interesting to find a $2$-knot $\Sigma$ with $\Delta(\Sigma)=1$ that fails to satisfy \eqref{eq: dknot-symmetric} and therefore does not admit a rational homology sphere Seifert surface. (We do not know of any other such obstruction.)
\end{remark}

\bibliography{bibliography}

\def\cprime{$'$}
\providecommand{\bysame}{\leavevmode\hbox to3em{\hrulefill}\thinspace}
\begin{thebibliography}{10}

\bibitem{BehrensGollaCorrection}
Stefan Behrens and Marco Golla, \emph{Heegaard {F}loer correction terms, with a
  twist}, to appear in Quantum Topol., \arxiv{1505.07401}, 2015.

\bibitem{CappellShaneson4Manifolds}
Sylvain~E. Cappell and Julius~L. Shaneson, \emph{Some new four-manifolds}, Ann.
  of Math. (2) \textbf{104} (1976), no.~1, 61--72.

\bibitem{cappell-shaneson:knots}
\bysame, \emph{There exist inequivalent knots with the same complement}, Ann.
  of Math. (2) \textbf{103} (1976), no.~2, 349--353.

\bibitem{donaldson:orientation}
Simon~K. Donaldson, \emph{The orientation of {Yang--Mills} moduli spaces and
  4-manifold topology}, {J. Diff.\ Geo.} \textbf{26} (1987), 397--428.

\bibitem{froyshov:qhs-floer}
Kim~A. Fr{\o}yshov, \emph{Monopole {F}loer homology for rational homology
  3-spheres}, Duke Math. J. \textbf{155} (2010), no.~3, 519--576.

\bibitem{GordonReversibility}
Cameron~McA. Gordon, \emph{On the reversibility of twist-spun knots}, J. Knot
  Theory Ramifications \textbf{12} (2003), no.~7, 893--897.

\bibitem{hatcher}
Allen Hatcher, \emph{Algebraic topology}, Cambridge University Press,
  Cambridge, 2002.

\bibitem{HeddenKimLivingston}
Matthew Hedden, Se-Goo Kim, and Charles Livingston, \emph{Topologically slice
  knots of smooth concordance order two}, J. Differential Geom. \textbf{102}
  (2016), no.~3, 353--393.

\bibitem{HeddenNiUnlink}
Matthew Hedden and Yi~Ni, \emph{Khovanov module and the detection of unlinks},
  Geom. Topol. \textbf{17} (2013), no.~5, 3027--3076.

\bibitem{hughes-ranicki:ends}
Bruce Hughes and Andrew Ranicki, \emph{Ends of complexes}, Cambridge Tracts in
  Mathematics, vol. 123, Cambridge University Press, Cambridge, 1996.

\bibitem{milnor-husemoller}
D.~Husemoller and J.~Milnor, \emph{Symmetric bilinear forms}, Springer-Verlag,
  Berlin and New York, 1973.

\bibitem{JabukaMarkProduct}
Stanislav Jabuka and Thomas~E. Mark, \emph{Product formulae for
  {Ozsv\'ath}--{Szab\'o} 4--manifold invariants}, Geom. Topol. \textbf{12}
  (2008), 1557--1651.

\bibitem{kirby:4-manifolds}
Robion~C. Kirby, \emph{The topology of {$4$}-manifolds}, Lecture Notes in
  Mathematics, vol. 1374, Springer-Verlag, Berlin, 1989.

\bibitem{kronheimer-mrowka:monopole}
P.~B. Kronheimer and T.~S. Mrowka, \emph{{Monopoles and Three-Manifolds}},
  Cambridge University Press, Cambridge, UK, 2008.

\bibitem{laitinen:ends}
Erkki Laitinen, \emph{End homology and duality}, Forum Math. \textbf{8} (1996),
  no.~1, 121--133.

\bibitem{LevineRubermanCorrection}
Adam~S. Levine and Daniel Ruberman, \emph{Generalized {H}eegaard {F}loer
  correction terms}, Proceedings of {G}\"okova {G}eometry-{T}opology
  {C}onference 2013, G\"okova Geometry/Topology Conference (GGT), G\"okova,
  2014, pp.~76--96.

\bibitem{LevineRubermanStrleNonorientable}
Adam~S. Levine, Daniel Ruberman, and Sa\v{s}o Strle, \emph{Nonorientable
  surfaces in homology cobordisms}, Geom. Topol. \textbf{19} (2015), no.~1,
  439--494, with an appendix by Ira M. Gessel.

\bibitem{LidmanInfinity}
Tye Lidman, \emph{On the infinity flavor of {H}eegaard {F}loer homology and the
  integral cohomology ring}, Comment. Math. Helv. \textbf{88} (2013), no.~4,
  875--898.

\bibitem{ManolescuOzsvathLink}
Ciprian Manolescu and Peter~S. Ozsv{\'a}th, \emph{Heegaard {F}loer homology and
  integer surgeries on links}, \arxiv{1011.1317}, 2010.

\bibitem{MilnorCyclic}
John~W. Milnor, \emph{Infinite cyclic coverings}, Conference on the {T}opology
  of {M}anifolds ({M}ichigan {S}t. {U}., 1967) (J.~G. Hocking, ed.), Prindle,
  Weber, and Schmidt, 1968, pp.~115--133.

\bibitem{mrowka-ruberman-saveliev:sw-index}
Tomasz Mrowka, Daniel Ruberman, and Nikolai Saveliev, \emph{{Seiberg-Witten
  equations, end-periodic Dirac operators, and a lift of Rohlin's invariant}},
  J. Differential Geom. \textbf{88} (2011), 333--377.

\bibitem{munkres:at}
James~R. Munkres, \emph{Elements of algebraic topology}, Addison-Wesley
  Publishing Company, Menlo Park, CA, 1984.

\bibitem{NiWuCosmetic}
Yi~Ni and Zhongtao Wu, \emph{Cosmetic surgeries on knots in {$S^3$}}, J. Reine
  Angew. Math. \textbf{706} (2015), 1--17.

\bibitem{novikov}
S.~P. Novikov, \emph{Topological invariance of rational classes of
  {P}ontrjagin}, Dokl. Akad. Nauk SSSR \textbf{163} (1965), 298--300.

\bibitem{OSzAbsolute}
Peter~S. Ozsv{\'a}th and Zolt{\'a}n Szab{\'o}, \emph{Absolutely graded {F}loer
  homologies and intersection forms for four-manifolds with boundary}, Adv.
  Math. \textbf{173} (2003), no.~2, 179--261.

\bibitem{OSzKnot}
\bysame, \emph{Holomorphic disks and knot invariants}, Adv. Math. \textbf{186}
  (2004), no.~1, 58--116.

\bibitem{OSzProperties}
\bysame, \emph{Holomorphic disks and three-manifold invariants: properties and
  applications}, Ann. of Math. (2) \textbf{159} (2004), no.~3, 1159--1245.

\bibitem{OSz4Manifold}
\bysame, \emph{Holomorphic triangles and invariants for smooth four-manifolds},
  Adv. Math. \textbf{202} (2006), no.~2, 326--400.

\bibitem{OSzSurgery}
\bysame, \emph{Knot {F}loer homology and integer surgeries}, Algebr. Geom.
  Topol. \textbf{8} (2008), no.~1, 101--153.

\bibitem{OSzRational}
\bysame, \emph{Knot {F}loer homology and rational surgeries}, Algebr. Geom.
  Topol. \textbf{11} (2011), no.~1, 1--68.

\bibitem{rohlin:pontrjagin}
V.~A. Rohlin, \emph{On {P}ontrjagin characteristic classes. ({R}ussian)}, Dokl.
  Akad. Nauk SSSR (N.S.) \textbf{113} (1957), 276--279.

\bibitem{Rolfsen}
Dale Rolfsen, \emph{Knots and links}, third ed., AMS Chelsea, Providence, 2003.

\bibitem{ruberman:ds}
Daniel Ruberman, \emph{Doubly slice knots and the {C}asson-{G}ordon
  invariants}, Trans. Amer. Math. Soc. \textbf{279} (1983), no.~2, 569--588.

\bibitem{ruberman-saveliev:casson}
Daniel Ruberman and Nikolai Saveliev, \emph{Rohlin's invariant and gauge
  theory. {I}. {H}omology 3-tori}, Comment. Math. Helv. \textbf{79} (2004),
  no.~3, 618--646.

\bibitem{ruberman-saveliev:mappingtori}
\bysame, \emph{Rohlin's invariant and gauge theory. {II}. {M}apping tori},
  Geom. Topol. \textbf{8} (2004), 35--76 (electronic).

\bibitem{ruberman-saveliev:survey}
\bysame, \emph{{C}asson--type invariants in dimension four}, {G}eometry and
  {T}opology of {M}anifolds, Fields Institute Communications, vol.~47, AMS,
  2005.

\bibitem{ruberman-saveliev:4tori}
\bysame, \emph{Rohlin's invariant and gauge theory. {III}. {H}omology 4--tori},
  Geom. Topol. \textbf{9} (2005), 2079--2127 (electronic).

\bibitem{weibel:intro}
Charles~A. Weibel, \emph{An introduction to homological algebra}, Cambridge
  Studies in Advanced Mathematics, vol.~38, Cambridge University Press,
  Cambridge, 1994.

\bibitem{YanagawaRibbon1}
Takaaki Yanagawa, \emph{On ribbon {$2$}-knots. {T}he {$3$}-manifold bounded by
  the {$2$}-knots}, Osaka J. Math. \textbf{6} (1969), 447--464.

\end{thebibliography}
\bibliographystyle{amsplain}

\end{document}